\documentclass[10pt]{amsart}
\usepackage{amsmath}
\usepackage{amssymb}
\usepackage{enumerate}
\usepackage{amsbsy}
\usepackage{amsfonts}
\usepackage{color}
\usepackage{upgreek}

\newtheorem{thm}{Theorem}[section]
\newtheorem{lem}[thm]{Lemma}
\newtheorem{prop}[thm]{Propsition}
\newtheorem{cor}[thm]{Corollary}
\newtheorem{defn}[thm]{Definition}
\newtheorem{rem}[thm]{Remark}

\numberwithin{equation}{section}

\begin{document}

	\title[cubic Dirac equations]{Scattering of cubic Dirac equations with a general class of Hartree-type nonlinearity for the critical Sobolev data}

     \author[S. Hong]{Seokchang Hong}
    \address{Department of Mathematics, Chung-Ang University, Seoul 06974, Korea}
    \email{seokchangh11@cau.ac.kr}

	
	\begin{abstract}
		Recently low-regularity behaviour of solutions to cubic Dirac equations with the Hartree-type nonlinearity has been extensively studied in somewhat a specific assumption on the structure of the nonlinearity. The key approach of previous results was to exploit the null structure in the nonlinearity and the decay of the Yukawa potential. In this paper, we aim to go beyond; we investigate the strong scattering property of cubic Dirac equations with quite a general class of the Hartree-type nonlinearity, which covers the Coulomb potential as well as the Yukawa potential, and the bilinear form, in which one cannot use the specific null structure. As a direct application, we also obtain the scattering for the boson-star equations with the scaling-critical Sobolev data.  
	\end{abstract}

		\maketitle

\section{Introduction}
In 1928, the Dirac equation was derived by P. A. M. Dirac \cite{D} to explain the behaviour of relativistic particles of spin $\frac12$, i.e., fermions.  His trial was successful; his equation is consistent with quantum physics and special relativity. Since then, not only has the Dirac equation shattered the light on the conjunction between quantum mechanics and the theory of special relativity, but it also reveals a new class of mathematical objects, which is of great interest in dispersive equations \cite{caccifesta,cacciafesta1,candy,escobedo,machiharaet,selbtes}. In the mathematical aspect one remarkable difference between the Dirac equation and other equations in quantum physics such as the Schr\"odinger or the Klein-Gordon equations is that the Dirac equation is not written as a single equation. Indeed, it is a linear system of four\footnote{For simplicity of our discussion we restrict ourselves in the $(1+3)$ dimensional setting.} coupled first-order partial differential equations. To be precise, the Dirac equation is often written in the single symbolic form as (in free particle case with non-negative constant $M\ge0$)
\begin{align}\label{homo-dirac}
-i\gamma^\mu\partial_\mu\psi+M\psi =0,	
\end{align}
where the unknown function $\psi$ is the complex-valued four-column field, which is referred to as the Dirac spinor field in the Minkowski space $(\mathbb R^{1+3},\mathbf m)$, where the metric $\mathbf m$ is given by $\mathbf m_{\mu\nu}=\textrm{diag}(-1,1,1,1)$. Here the $\gamma^\mu$, $\mu=0,1,2,3$ are the $4\times4$ complex-valued matrices 
\begin{align*}
\gamma^0 = \begin{bmatrix} 
 	I_{2\times2} & \mathbf0 \\ \mathbf0 & -I_{2\times2}
 \end{bmatrix}, \ \gamma^j = \begin{bmatrix} 
 	\mathbf 0 & \sigma^j \\ -\sigma^j & \mathbf0 
 \end{bmatrix}	,
\end{align*}
where the Pauli matrices $\sigma^j$, $j=1,2,3$ are the $2\times2$ complex-valued matrices, given by
\begin{align*}
\sigma^1 = \begin{bmatrix} 
 	0 & 1 \\ 1 & 0 
 \end{bmatrix}, \ \sigma^2 = \begin{bmatrix} 
 	0 & -i \\ i & 0 
 \end{bmatrix}, \  \sigma^3=\begin{bmatrix} 
 	1 & 0 \\ 0 & -1
 \end{bmatrix},	
\end{align*}
and $I_{n\times n}$ is the $n\times n$ identity matrix.
We refer to Section \ref{sec:dirac-op} for some algebraic structures of gamma $\gamma^\mu$ and Pauli matrices $\sigma^j$.
We let $(x^\mu)$, $\mu=0,\cdots,3$ denote points in the Minkowski space. The partial derivatives with respect to $x^\mu$ is written by $\partial_\mu$. We shall use the notation $t=x^0$ for time variable, and $x=(x^1,x^2,x^3)$ for spatial variable. Then we write $\partial_0=\partial_t$ and $\nabla=(\partial_1,\partial_2,\partial_3)$. 
Throughout this paper we adapt the summation convention, i.e., any repeated indices mean the summation over described range. For example we write $\gamma^\mu\partial_\mu = \gamma^0\partial_t+\sum_{j=1}^3\gamma^j\partial_j$. 

Based on the homogeneous Dirac equations \eqref{homo-dirac}, the Dirac equation with specific nonlinearities has been extensively studied to model the self-interacting Dirac fermions, such as electrons. The cubic Dirac equations is one of a widely considered toy model
\begin{align}\label{cubic-dirac1}
-i\gamma^\mu\partial_\mu\psi+M\psi = (\psi^\dagger\gamma^0\Gamma\psi)\Gamma\psi. 	
\end{align}
 The equation \eqref{cubic-dirac1} is called the Soler model \cite{soler} when $\Gamma=I_{4\times4}$ and the Thirring model \cite{thirring} if $\Gamma=\gamma^\mu$. Here the dagger notation $\psi^\dagger$  stands for the complex conjugate transpose of $\psi$, i.e., $\psi^\dagger=(\psi^*)^T$. The non-negative constant $M\ge0$ is a mass of the described fermion. Low regularity behavior of solutions to cubic Dirac equations in two or three dimensional setting is well-known. We refer the readers to \cite{behe1,behe2,boucan,machiharaet,pecher}. See also \cite{caccifesta} for more general cubic nonlinearity.
 
One may also consider a physical system in which the Dirac fermion is coupled with a scalar field. In this paper we only present the Dirac-Klein-Gordon system which is obtained by coupling the Dirac field and the Klein-Gordon field
\begin{align}\label{intro-gen-dkg}
\begin{aligned}
-i\gamma^\mu\partial_\mu\psi+M\psi &= g\varphi\Gamma\psi , \\
(\Box+m^2)\varphi &= -g\psi^\dagger\gamma^0\Gamma\psi,	
\end{aligned}
\end{align}
which describes an interaction of the Dirac fermion with the meson field \cite{bjor}. The complex matrix $\Gamma\in\mathbb C^{4\times4}$ is chosen to be $I_{4\times4}, \gamma^0$, or $\gamma^5$ by the preference of researchers and $g$ is the coupling constant. From the system \eqref{intro-gen-dkg}, one can derive cubic Dirac equations with the Hartree-type nonlinearity \cite{chagla,tes1,cyang}, given by
\begin{align}\label{main-hartree}
\left\{
\begin{array}{l}
	-i\gamma^\mu\partial_\mu\psi +M\psi = gV_b*(\psi^\dagger\gamma^0\Gamma\psi)\Gamma\psi, \\
	\psi|_{t=0} = \psi_0.
\end{array}
\right.	
\end{align}
The potential $V_b$ is the spatial potential
$$
V_b(x) = \frac1{4\pi}\frac{e^{-b|x|}}{|x|}, \quad b\ge0,
$$
and in the sequel we put the coupling constant $g=1$ for simplicity. 
In particular, the potential $V_b$ is the Coulomb-type when $b=0$, whereas it is the Yukawa-type for $b>0$. Note that the Fourier transform of the potential $V_b$ is $\sqrt{b^2+|\xi|^2}^{-2} = \langle\xi\rangle_b^{-2}$. The equation \eqref{main-hartree} with $M=0$ and $b=0$ obeys the scale-invariant symmetry, i.e., for a solution $\psi$ to the equations \eqref{main-hartree}, $\psi_\tau(t,x) = \tau^\frac32\psi(\tau t, \tau x),\, \tau>0$ is also the solutions to \eqref{main-hartree} and the scale-invariant Sobolev space is $L^2_x$. The long-time behavour of solutions to the equation \eqref{main-hartree} for a low regularity data is also well-studied especially when $\Gamma=I_{4\times4}$, $M>0$, and $b>0$ \cite{cholee,choozlee,tes,tes1,cyang}. We only mention one result among the previous works. We define the angularly regular homogeneous Sobolev spaces to be the set $\dot H^{s,\sigma}$ whose norm is given by $\|\langle\Omega\rangle^\sigma f\|_{\dot H^s}$, where $\langle\Omega\rangle=(1-\Delta_{\mathbb S^2})^\frac12$ and $\Delta_{\mathbb S^2}$ is the Laplace-Beltrami operator on the unit sphere $\mathbb S^2\subset\mathbb R^3$. We also define the inhomogeneous space $H^{s,\sigma}$ in the obvious way. When $s=0$, we denote it by $L^{2,\sigma}$.
\begin{thm}[Theorem 1.1 of \cite{chohlee}]\label{gwp-hartree3}
Let $M>0$ and $b>0$. We consider the equations \eqref{main-hartree} with $\Gamma=I_{4\times4}$. We let $\sigma>0$. Suppose that the initial data $\psi_0\in L^{2,\sigma}(\mathbb R^3)$ satisfies
$$
\|\psi_0\|_{L^{2,\sigma}} \ll1.
$$	
Then the Cauchy problems for the equation \eqref{main-hartree} is globally well-posed and the solutions scatter to free waves. 
\end{thm} 
The author of \cite{cyang} obtained the failure of $C^3$-smoothness of the solution maps in the supercritical range $s<0$. In this aspect, Theorem \ref{gwp-hartree3} is optimal up to a small amount of extra weighted regularity in angular variables.
The natural question is then whether solutions to the equations \eqref{main-hartree} obey scattering property when $b=0$ or one chooses other than $\Gamma=I_{4\times4}$. For $\Gamma=\gamma^0$ and $b=0$, and $M>0$, the answer to the question is in general negative \cite{chooz,chohlee1}. Instead, the authors of \cite{cho} obtained modified scattering for the case $\Gamma=\gamma^0$ and $b=0$, and $M>0$. We also refer to \cite{geosha} for global solutions to the equations with $\Gamma=\gamma^0$ in two dimensional setup. 

Motivated by partially positive or negative answers, we aim to investigate strong scattering property of solutions to cubic Dirac equations with a wide class of the Hartree-type nonlinearity. We hope to establish global solutions to the equations \eqref{main-hartree} with more general cases, in which one cannot exploit the null structure in the nonlinearity, or one may encounter a serious singularity near origin. To elucidate this point, it is instructive to discuss the main difference between the Yukawa and Coulomb potentials, and between the choices of $\Gamma$, especially $\Gamma=I_{4\times4}$ and $\Gamma=\gamma^0$.
\subsection{Coulomb potential and Yukawa potential}
We recall the potential $V_b(x)=\frac1{4\pi}\frac{e^{-b|x|}}{|x|}$ for $x\in\mathbb R^3$. For $b\ge0$, the Fourier transform of $V_b$ is given by $(b^2+|\xi|^2)^{-1}$.
Compared to the Yukawa potential, the Fourier transform of the Coulomb potential possesses the singularity at the origin. This requires one to deal with carefully the low-output frequency. More precisely, one need to work on the bilinear estimates
$$
\|P_\mu H_N(P_{\lambda_1}H_{N_1}\varphi^\dagger\gamma^0P_{\lambda_2}H_{N_2}\phi)\|_{L^2_tL^2_x},
$$
with low frequency: $0<\mu\le1$. Here $P_\lambda$ is the usual Littlewood-Paley projections onto the set $\{\xi\in\mathbb R^3: |\xi|\approx\lambda \}$ and $H_N$ denotes the projections on angular frequencies of size $N$. Note that for $\mu\gtrsim1$, the above bilinear estimates have been already known in \cite{chohlee,chohlee1}, concerning the scattering property with the Yukawa potential, and there is no essential difference between the Coulomb potential and Yukawa potential in the Fourier side when the output frequency is not very small, i.e., $\mu\gtrsim1$. One should also remark that the output frequency $\mu$ can be such a \textit{low} frequency as $\mu\lesssim1$  only if two input frequencies are high and cause resonant interactions. Meanwhile, it is easy to see that the high-output frequency case, i.e., $\mu\approx\lambda_1$ or $\mu\approx\lambda_2$, is rather easier than low-output case, since the Fourier multiplier $\widehat{V_b}$ is a kernel which exhibits good decay. In consequence, the crucial part of the above bilinear estimates is the low-output case, especially when $\mu\lesssim1$ and $\mu\ll\lambda_1\approx\lambda_2$.   
\subsection{Null structure}\label{intro-null}
Now we give a remark on the nonlinearity in \eqref{main-hartree} with $\Gamma=I_{4\times4}$ and $\Gamma=\gamma^0$. The bilinear forms $\psi^\dagger\psi$ and $\psi^\dagger\gamma^0\psi$ exhibit totally different structures. We shall investigate the main difference between two products of spinor fields by exploiting the projection operators $\Pi_\theta^M$ for $\theta\in\{+,-\}$, defined as the Fourier multiplier $\widehat{\Pi_\theta^M f}(\xi)=\Pi_\theta^M(\xi)\widehat f(\xi)$ with
\begin{align}
\Pi_\theta^M(\xi) = \frac12\left(I_{4\times4}+\frac{\theta}{\langle\xi\rangle_M}(\xi_j\gamma^0\gamma^j+M\gamma^0) \right),	
\end{align}
where $\langle\xi\rangle_M = \sqrt{|\xi|^2+M^2} $, for $M\ge0$.
An easy computation shows that $\Pi_\theta^M\Pi_{-\theta}^M=0$ and $\Pi_\theta^M\Pi_\theta^M=\Pi_\theta^M$. One can also obtain the commutator identity: (see \cite{behe})
\begin{align}\label{comm-massive}
\gamma^0\Pi_\theta^M = \Pi_{-\theta}^M\gamma^0+\theta \frac{M}{\langle\xi\rangle_M}.
\end{align}
Using these one may observe that the product $\psi^\dagger\gamma^0\psi$ is written in the Fourier side as 
\begin{align*}
\mathcal F_x(\psi^\dagger\gamma^0\psi)& = \sum_{\theta_1,\theta_2\in\{+,-\}}\mathcal F_x[(\Pi_{\theta_1}^M\psi)	^\dagger\gamma^0(\Pi_{\theta_2}^M\psi)](\xi) \\
& = \sum_{\theta_1,\theta_2\in\{+,-\}} \int_{\mathbb R^3}\widehat{\psi^\dagger}(\eta)\Pi_{\theta_1}^M(\eta)\Pi_{-\theta_2}^M(\xi-\eta)\widehat{\psi}(\xi-\eta)\,d\eta. \\
&\qquad + \theta_2M\mathcal F_x[(\Pi_{\theta_1}^M\psi)^\dagger\langle\nabla\rangle_M^{-1}\psi](\xi).
\end{align*}
We notice that the composition of projections $\Pi_{\theta_1}\Pi_{\theta_2}$ yields an additional cancellation property as 
\begin{align}\label{int-dirac-null}
|\Pi_{\theta}^M(\xi)\Pi_{-\theta}^M(\eta)| \lesssim \angle(\xi,\eta)+M(\langle\xi\rangle_M^{-1}+\langle\eta\rangle_M^{-1}),
\end{align}
for $\xi,\eta\in\mathbb R^3$. (See \cite[Lemma 3.1]{behe}.)
In consequence, we observe that the bilinear form in the nonlinearity of the equation \eqref{main-hartree} with $\Gamma=I_{4\times4}$ possesses the cancellation property. The very existence of the null structure is of great importance in the analysis in view of our previous observation in \cite{chohlee}:
\begin{thm}[Theorem 1.2 of \cite{chohlee}]\label{nonscatter-thm}
Let $b=0$ and $M>0$. 
Assume that $\psi$ be a smooth solution to \eqref{main-hartree} with $\Gamma=\gamma^{0}$ which scatters in $L_x^2$ to a smooth solution $\psi_\infty^\ell$. 
Then $\psi,\psi_\infty^\ell=0$ in $L_x^2$.
\end{thm}
\subsection{Additional angular regularity}
We would like to highlight the reason why we need an additional angular regularity even when one deals with the equations \eqref{main-hartree} with $\Gamma=I_{4\times4}$, in which case one can exploit the null-form-bound. The above discussion on the null structure says that one can obtain an additional cancellation property from the bilinear form $(\Pi_{\theta_1}^M\varphi)^\dagger\gamma^0(\Pi_{\theta_2}^M\psi)$ when $\theta_1=\theta_2$. However, the situation becomes totally different if $\theta_1\neq\theta_2$. 
\begin{prop}[Proposition 3.7 of \cite{cyang}]
Let $\mu\ll\lambda_1\approx\lambda_2$. We use the shorthand $P_\lambda\psi=\psi_\lambda$. If $\theta_1\neq\theta_2$, then
\begin{align}\label{hhl-trans}
\|P_\mu [(\Pi_{\theta_1}^M\varphi_{\lambda_1})^\dagger\gamma^0(\Pi_{\theta_2}^M\psi_{\lambda_2})]\|_{L^2_tL^2_x(\mathbb R^{1+3})} & \lesssim \mu \|\varphi_{\lambda_1}\|_{V^2_{\theta_1}}\|\psi_{\lambda_2}\|_{V^2_{\theta_2}}	.
\end{align}
\end{prop}
We refer the readers to Section 2 for the definition of the $V^2_\theta$-spaces. The situation goes even worse when we are concerned with the specific cubic nonlinearity. Indeed, in view of an energy inequality (see Lemma \ref{energy-ineq}) we need to consider the following integral:
\begin{align}
\int_{\mathbb R^{1+3}}V_b*(\psi_{\lambda_1}^\dagger\gamma^0\psi_{\lambda_2})	(\psi_{\lambda_4}^\dagger\gamma^0\psi_{\lambda_3})\,dxdt,
\end{align}
where we omit the projection operators $\Pi_{\theta_j}$ for brevity. We let $\xi_j$ denote the spatial Fourier variables of the spinors $\psi_j$, $j=1,\cdots,4$, respectively. We also denote the temporal Fourier variables of the spinors $\psi_j$ by $\tau_j$. Then we have the relations
$$
-\xi_1+\xi_2-\xi_4+\xi_3 =0, \textrm{ and } -\tau_1+\tau_2-\tau_4+\tau_3=0.
$$ 
We define the modulation $|\tau+\theta\langle\xi\rangle_M|$, the distance to the characteristic hypersurface. Now we define the modulation function
\begin{align}\label{intro-modulation}
\mathcal M_{\theta_{1234}}(\xi_1,\xi_2,\xi_3,\xi_4)=| -\theta_1\langle\xi_1\rangle_M+\theta_2\langle\xi_2\rangle_M+\theta_3\langle\xi_3\rangle_M-\theta_4\langle\xi_4\rangle_M |	.
\end{align}
The relative size of the modulation function is of great importance in the practical analysis. Indeed, the largeness of $\mathcal M_{\theta_{1234}}$ implies an osciallation, and then one can expect a certain cancellation. However, when the function $\mathcal M_{\theta_{1234}}$ is small, which we call the resonant interactions, we cannot expect such an oscillation, and this case will be our main concern. In view of \eqref{intro-modulation}, the resonant interactions can occur when $\theta_1=\theta_2$ and $\theta_3=\theta_4$ with $\lambda_1\approx\lambda_2$ and $\lambda_3\approx\lambda_4$, in which case one can enjoy the null structure. However, the resonant interactions can also occur in the case that one cannot exploit the null structure. In fact, if $(\theta_1,\theta_2,\theta_3,\theta_4)=(+,-,-,+)$ or $(-,+,+,-)$ with $\lambda_1\approx\lambda_2\approx\lambda_3\approx\lambda_4$, then the function $\mathcal M_{\theta_{1234}}$ can be quite small even when the modulation of the spinors $\psi_j$ is high. One can at most obtain the bound as \eqref{hhl-trans}. It turns out that this case is the most critical case in the analysis of the equations \eqref{main-hartree} with $\Gamma=I_{4\times4}$, which impedes one from attaining the scattering results for the scaling critical Sobolev data.

Now the role of an additional angular regularity seems obvious. By imposing angular regularity we will obtain the improved bound as (see Lemma \ref{main-bi})
\begin{align}\label{hhl-ang-int}
\|P_\mu H_N(\varphi_{\lambda_1,N_1}^\dagger\gamma^0\psi_{\lambda_2,N_2})\|_{L^2_tL^2_x} & \lesssim \mu\left(\frac\mu{\min\{\lambda_1,\lambda_2\}}\right)^{\mathfrak d}	\min\{N_1,N_2\}\|\varphi_{\lambda_1,N_1}\|_{V^2_{\theta_1}}\|\psi_{\lambda_2,N_2}\|_{V^2_{\theta_2}},
\end{align}
where $\mathfrak d$ is some positive number slightly smaller than $\frac14$. Furthermore, we will see that the relative size of modulation is not important in the proof of the estimates \eqref{hhl-ang-int}. A simple combination of the bounds \eqref{hhl-trans} and \eqref{hhl-ang-int} gives for a small $\delta\ll1$,
\begin{align}
	\|P_\mu H_N(\varphi_{\lambda_1,N_1}^\dagger\gamma^0\psi_{\lambda_2,N_2})\|_{L^2_tL^2_x} & \lesssim \mu\left(\frac\mu{\min\{\lambda_1,\lambda_2\}}\right)^{\delta}	(\min\{N_1,N_2\})^\delta\|\varphi_{\lambda_1,N_1}\|_{V^2_{\theta_1}}\|\psi_{\lambda_2,N_2}\|_{V^2_{\theta_2}},
\end{align}
which turns out to be enough bound to obtain the desired global solutions. 

\subsection{Main results}
In this paper, we establish the scattering property for the equation \eqref{main-hartree} at the scaling critical regularity in the mass-less case with the Coulomb-type potential, as well as the Yukawa-type potential. In other words, we are concerned with the cubic Dirac equation \eqref{main-hartree} with $M=0$ and $b\ge0$ for the scale-invariant Sobolev data. 
\begin{thm}\label{gwp-hartree1}
Let $M=0$ and $b\ge0$. We consider the equations \eqref{main-hartree} with $\Gamma=I_{4\times4}$. We let $\sigma>0$. Suppose that the initial data $\psi_0\in L^{2,\sigma}(\mathbb R^3)$ satisfies
$$
\|\psi_0\|_{L^{2,\sigma}} \ll1.
$$	
Then the Cauchy problems for the equation \eqref{main-hartree} is globally well-posed and the solutions scatter to free waves. 
\end{thm}
Theorem \ref{gwp-hartree1} presents the positive answers on the scattering problems of the equations \eqref{main-hartree} with both Yukawa-type potential and Coulomb-type potential in the mass-less case. Now we pay attention to the equations with $\Gamma=\gamma^0$, whose nonlinearity does not possess the null structure as the case $\Gamma=I_{4\times4}$. In the massive case, i.e., $M>0$ the authors of \cite{chohlee} observed non-scattering for $b=0$. However, we get the positive answer on the scattering problem in the mass-less case even when $b=0$. 
\begin{thm}\label{gwp-hartree2}
Let $M=0$ and $b\ge0$. We consider the equations \eqref{main-hartree} with $\Gamma=\gamma^0$. We let $\sigma=1$. Suppose that the initial data $\psi_0\in L^{2,\sigma}(\mathbb R^3)$ satisfies
$$
\|\psi_0\|_{L^{2,\sigma}} \ll1.
$$	
Then the Cauchy problems for the equation \eqref{main-hartree} is globally well-posed and the solutions scatter to free waves. 
\end{thm}
The main drawback of the equations with $\Gamma=\gamma^0$ is that one cannot take an advantage of the null structure, which can relax the resonant interactions arising from especially the High $\times$ High $\Rightarrow$ Low interactions. To overcome the problem we exploit fully an angular momentum operator, which turns out to be the efficient equipment to eliminate such a bad interaction. Indeed, the use of an improved space-time estimate by spending an additional angular regularity (Proposition \ref{stri-ang}), and the application of an angular concentration phenomena (Lemma \ref{ang-con}) show that a certain amount of an angular regularity can substitute for the role of null structures in some sense.   

Furthremore, in the proof of Theorem \ref{gwp-hartree2} we see that the identical argument implies that one can obtain scattering results for the equations \eqref{main-hartree} with $\Gamma=\gamma^0$ and $M>0$, and $b>0$.
\begin{thm}\label{gwp-hartree4}
Let $M>0$ and $b>0$. We consider the equations \eqref{main-hartree} with $\Gamma=\gamma^0$. We let $\sigma=1$. Suppose that the initial data $\psi_0\in L^{2,\sigma}(\mathbb R^3)$ satisfies
$$
\|\psi_0\|_{L^{2,\sigma}} \ll1.
$$	
Then the Cauchy problems for the equation \eqref{main-hartree} is globally well-posed and the solutions scatter to free waves. 
\end{thm}
We postpone the main strategy of the proof of the main Theorem to Section 3. Instead, we summarize the scattering results for the cubic Dirac equations with the Hartree-type nonlinearity on $\mathbb R^{1+3}$. 
\begin{align}\label{table1}
\begin{tabular}{||c|c|c||}
\hline
	$\Gamma = I_{4\times4}$ & $M>0$ & $M=0$ \\
	\hline
	$b>0$ & Theorem \ref{gwp-hartree3} & Theorem \ref{gwp-hartree1} \\
	\hline
	$b =0$ & open  & Theorem \ref{gwp-hartree1}\\
	\hline
\end{tabular}	
\end{align}
The table \eqref{table1} shows that only massive case with the Coulomb potential is open. 
The table \eqref{table2} below shows the scattering results for the case $\Gamma=\gamma^0$. 
\begin{align}\label{table2}
\begin{tabular}{||c|c|c||}
\hline
	$\Gamma = \gamma^0$ & $M>0$ & $M=0$ \\
	\hline
	$b>0$ & Theorem \ref{gwp-hartree4} & Theorem \ref{gwp-hartree2} \\
	\hline
	$b =0$ & Theorem \ref{nonscatter-thm}  &  Theorem \ref{gwp-hartree2} \\
	\hline
\end{tabular}	
\end{align}
Recently the authors of \cite{chohoz} studied the scattering problems for the equation \eqref{main-hartree} for the case $\Gamma=\gamma^5=i\gamma^0\gamma^1\gamma^2\gamma^3$. It turns out that the commutator identity $\gamma^0\gamma^5\Pi_\theta^M=\Pi_{-\theta}^M\gamma^0\gamma^5$ holds and hence we can utilise the null structure in the same way as the case $\Gamma=I_{4\times4}$.
\begin{align}\label{table3}
\begin{tabular}{||c|c|c||}
\hline
	$\Gamma = \gamma^5$ & $M>0$ & $M=0$ \\
	\hline
	$b>0$ & \cite{chohoz} & Theorem \ref{gwp-hartree1} \\
	\hline
	$b =0$ & open  &  Theorem \ref{gwp-hartree1} \\
	\hline
\end{tabular}	
\end{align}
\begin{rem}
We would like to discuss briefly the massive case, i.e., $M>0$.
When $\Gamma=I_{4\times4}$, we shall use the commutator identity \eqref{comm-massive} to exploit the null structure in the bilinear form $\psi^\dagger\gamma^0\psi$. If $b>0$, i.e., $V_b$ is the Yukawa-type potential, the second term in the right-handside of \eqref{comm-massive} is nothing but an error term. On the other hand, if $b=0$, i.e., $V_b$ is the Coulomb-type potential, the term $\langle\nabla\rangle_M^{-1}$ is not an error term anymore. (See also \eqref{int-dirac-null}.) When two high-input frequencies cause low-output frequency $\mu\lesssim1$, the additional term $\langle\nabla\rangle_M^{-1}$ remains problematic; one cannot relax the whole singularity $|\xi|^{-2}$ even though one utilise the null structure. This is why answers on the scattering problem for the massive case with $\Gamma=I_{4\times4}$ or $\Gamma=\gamma^5$ and $b=0$ are still open. Therefore, it will be left to our future work.
\end{rem}
\begin{rem}
As small data scattering is established, the natural question is whether arbitrarily large data scattering is also possible. It is still open, however, it is partially positive provided that one impose a certain condition on initial datum \cite{chohlee1}. Indeed, one approach is to define a controlling space-time Lebesgue norm and obtain bounded solutions which exists globally in time and scatter, as the controlling norm remains bounded. On the other hand, one can also apply charge conjugation approach (or the Majorana condition \cite{majorana}). In the several cases of various choices of $\Gamma\in\mathbb C^{4\times4}$, in which positive answers on the scattering problems of the equations hold at least for small initial data, one can follow the aforementioned approaches and establish conditional large data scattering. Since the proof is quite similar as \cite{candyherr1,chohlee1}, we only present in Appendix the proof of bilinear estimates, which is required to obtain the desired conditional-large data scattering.
\end{rem}
\subsection{The boson star equation}
We end this section with a brief introduction on a related equation.
We present the Cauchy problems for the boson star equation (or the semi-relativistic equation with the Hartree-type nonlinearity) on $\mathbb R^{1+3}$:
\begin{align}\label{boson-star}
\left\{
\begin{array}{l}
	-i\partial_tu + \sqrt{m^2-\Delta}u = (V_b*|u|^2)u, \\
	u|_{t=0} = u_0
\end{array}\right.	
\end{align}
We refer to \cite{chooz,herrlenz,herrtes,pusa} for this well-studied equation. After the use of the Dirac projection operators (see Section \ref{sec:dirac-op}) cubic Dirac equations \eqref{main-hartree} with $\Gamma=\gamma^0$ and $b>0$ is of the form \eqref{boson-star}. Thus as a direct application of Theorem \ref{gwp-hartree4}, we have the following. 
\begin{cor}\label{dirac-appli}
Suppose that $m>0$ and $b>0$. Let $\sigma=1$. Suppose that the initial data $u_0\in L^{2,\sigma}$ satisfies $\|u_0\|_{L^{2,\sigma}}\ll1$. The Cauchy problems for the equation \eqref{boson-star} is globally well-posed and scatters to free solutions as $t\rightarrow\pm\infty$.	
\end{cor}
\noindent By Corollary \ref{dirac-appli} we improve the previous results on the Cauchy problems for \eqref{boson-star} and attain the scaling critical regularity. We also refer to \cite{gioz}, which concerns Schr\"odinger equations with the Hartree-type nonlinearity.
\subsection*{Organisation}
In what follows, we present the basic notations. In Section 2 we give preliminaries, which involves the Dirac operators, analysis on the unit sphere, $U^p-V^p$ spaces, and several auxiliary estimates. Section 3 is devoted to the proof of our main results. We briefly discuss the proof of conditional large data scattering in Appendix. 
\subsection*{Notations}
\begin{enumerate}
\item
As usual different positive constants, which are independent of dyadic numbers $\mu,\lambda$, and $d$ are denoted by the same letter $C$, if not specified. The inequalities $A \lesssim B$ and $A \gtrsim B$ means that $A \le CB$ and
$A \ge C^{-1}B$, respectively for some $C>0$. By the notation $A \approx B$ we mean that $A \lesssim B$ and $A \gtrsim B$, i.e., $\frac1CB \le A\le CB $ for some absolute constant $C$. We also use the notation $A\ll B$ if $A\le \frac1CB$ for some large constant $C$. Thus for quantities $A$ and $B$, we can consider three cases: $A\approx B$, $A\ll B$ and $A\gg B$. In fact, $A\lesssim B$ means that $A\approx B$ or $A\ll B$.
\item 
The spatial and space-time Fourier transform are defined by
$$
\widehat{f}(\xi) = \int_{\mathbb R^3} e^{-ix\cdot\xi}f(x)\,dx, \quad \widetilde{u}(\tau,\xi) = \int_{\mathbb R^{1+3}}e^{-i(t\tau+x\cdot\xi)}u(t,x)\,dtdx.
$$
We also write $\mathcal F_x(f)=\widehat{f}$ and $\mathcal F_{t, x}(u)=\widetilde{u}$. We denote the backward and forward wave propagation of a function $f$ on $\mathbb R^3$ by
$$
e^{-\theta it |\nabla|}f = \frac1{(2\pi)^3}\int_{\mathbb R^3}e^{ix\cdot\xi}e^{-\theta it|\xi|}\widehat{f}(\xi)\,d\xi,
$$
where $\theta\in\{+,-\}$.
\item For any function $m$ on $\mathbb R^3$, we define $m(-i\nabla)$ as the Fourier multiplier operator with symbol $m(\xi)$, i.e., $\mathcal F_x[m(-i\nabla)f]=m(\xi)\widehat{f}(\xi)$. For example, recall that $V_b=\frac1{4\pi}\frac{e^{-b|x|}}{|x|}$. We write $V_b*f=\langle\nabla\rangle_b^{-2}f$.
\end{enumerate}

\section{Preliminaries}
\subsection{Dirac projection operators}\label{sec:dirac-op}
We let $(\mathbb R^{1+3},\mathbf m)$ be the Minkowski space with the metric $\mathbf m_{\mu\nu}=\textrm{diag}(-1,1,1,1)$. We first consider the gamma matrices $\gamma^\mu\in\mathbb C^{4\times4}$, $\mu=0,1,\cdots,3$, given by
\begin{align*}
\gamma^0 = \begin{bmatrix} 
 	I_{2\times2} & \mathbf0 \\ \mathbf0 & -I_{2\times2}
 \end{bmatrix}, \ \gamma^j = \begin{bmatrix} 
 	\mathbf 0 & \sigma^j \\ -\sigma^j & \mathbf0 
 \end{bmatrix}	,
\end{align*}
with the Pauli matrices $\sigma^j\in\mathbb C^{2\times2}$, $j=1,2,3$, given by
\begin{align*}
\sigma^1 = \begin{bmatrix} 
 	0 & 1 \\ 1 & 0 
 \end{bmatrix}, \ \sigma^2 = \begin{bmatrix} 
 	0 & -i \\ i & 0 
 \end{bmatrix}, \  \sigma^3=\begin{bmatrix} 
 	1 & 0 \\ 0 & -1
 \end{bmatrix}.	
\end{align*}
An easy computation implies that the Pauli matrices satisfy the following algebraic properties
$$
\sigma^j\sigma^k+\sigma^k\sigma^j = 2\delta^{jk}I_{2\times2}, \quad \sigma^1\sigma^2\sigma^3 = iI_{2\times2},
$$ 
where $\delta^{jk}$ is the usual Kronecker delta. Using these algebraic relation, we also obtain for $\gamma^\mu$:
\begin{align}\label{gamma-relation}
\gamma^\mu\gamma^\nu+\gamma^\nu\gamma^\mu = -2\mathbf m^{\mu\nu}I_{4\times4}.
\end{align}
Now we introduce the Dirac operator with a mass $M\ge0$
\begin{align}\label{dirac-operator}
\mathfrak D\psi = (-i\gamma^\mu\partial_\mu+M)\psi.	
\end{align}
Then the equation \eqref{main-hartree} is rewritten as
$$
\mathfrak D\psi - V_b*(\psi^\dagger\gamma^0\Gamma\psi)\Gamma\psi=0.
$$
\begin{rem}
	It is possible to relax the restriction on the gamma matrices. In fact, one only needs the assumption \eqref{gamma-relation} on the gamma matrices to study the Dirac operator. Nevertheless, we adapt the specific representative of the gamma matrices obtained by the Pauli matrices for the convenience. We refer the readers to \cite{ozya}.
\end{rem}
In the practical study of the Dirac equation, it is accessible to rewrite the equation as the half-wave equation. For this purpose, we introduce the projection operators for $\theta\in\{+,-\}$
\begin{align}\label{dirac-projection}
\Pi_\theta(\xi) = \frac12\left(I_{4\times4}+\theta\frac{\xi_j\gamma^0\gamma^j}{|\xi|} \right),	
\end{align}
where we used the summation convention.
Now we define the Fourier multiplier $\Pi_\theta$ by the identity $\mathcal F_x[\Pi_\theta f](\xi) = \Pi_\theta(\xi)\widehat{f}(\xi)$. By an easy computation one easily see the identity $\Pi_\theta\Pi_\theta=\Pi_\theta$ and $\Pi_\theta\Pi_{-\theta}=0$. An important identity of the projection $\Pi_\theta$ is the commutator identity with the gamma matrices, which presents $\gamma^0\Pi_\theta=\Pi_{-\theta}\gamma^0$.  We also have $\psi=\Pi_+\psi+\Pi_-\psi$. Then we see that 
$$
\Pi_\theta(\gamma^0\mathfrak D\psi)= (-i\partial_t+\theta|\nabla|)\Pi_\theta\psi -M\gamma^0\Pi_{-\theta}\psi.
$$
Note that in mass-less case, i.e., $M=0$, the above equation is nothing but a half-wave equation. This is the very first step of study on the dispersive property of the Dirac equation. Now we rewrite the equation \eqref{main-hartree} using the projections \eqref{dirac-projection} in the mass-less case. For $\Gamma=I_{4\times4}$, we obtain
\begin{align}\label{dirac-decom}
\left\{
\begin{array}{l}
	(-i\partial_t+\theta|\nabla|)\psi_\theta = \Pi_\theta[V_b*(\psi^\dagger\gamma^0\psi)\gamma^0\psi], \\
	\psi_\theta|_{t=0}=\psi_{0,\theta},
\end{array}
\right.	
\end{align}
 where $\psi_\theta = \Pi_\theta\psi$. Similarly, for $\Gamma=\gamma^0$ we get
 \begin{align}\label{dirac-decom1}
\left\{
\begin{array}{l}
	(-i\partial_t+\theta|\nabla|)\psi_\theta = \Pi_\theta[V_b*(\psi^\dagger\psi)\psi], \\
	\psi_\theta|_{t=0}=\psi_{0,\theta}.
\end{array}
\right.	
\end{align}
\begin{rem}
When we are concerned with massive Dirac equations, i.e., $M>0$, we need to invoke a slightly modified projection operator for $\theta\in\{+,-\}$
$$
\Pi^M_\theta (\xi) = \frac12\left( I_{4\times4}+\theta \frac{\xi_j\gamma^0\gamma^j+M\gamma^0}{\langle\xi\rangle_M} \right),
$$	
where $\langle\xi\rangle_M=\sqrt{M^2+|\xi|^2}$. Using the projection $\Pi_\theta^M$, we rewrite the equation \eqref{main-hartree} in the massive case as 
$$
(-i\partial_t+\theta\langle\nabla\rangle_M)\psi_\theta = \Pi_\theta^M[V_b*(\psi^\dagger\gamma^0\Gamma\psi)\gamma^0\Gamma\psi],
$$
which turns out to be of the form of half-wave decomposition of nonlinear Klein-Gordon equations. 
\end{rem}

\subsection{Multipliers}\label{multi}
We fix a smooth function $\rho\in C^\infty_0(\mathbb R)$ such that $\rho$ is supported in the set $\{ \frac12<t<2\}$ and we let
$$
\sum_{\lambda\in2^{\mathbb Z}}\rho\left(\frac t\lambda\right) =1.
$$
 We define $\mathcal Q_\mu$ to be a finitely overlapping collection of cubes of diameter $\frac{\mu}{1000}$ covering $\mathbb R^3$, and let $\{ \rho_\mathtt q\}_{\mathtt q\in\mathcal Q_\mu}$ be a corresponding subordinate partition of unity. Now we define the standard Littlewood-Paley multipliers, for $\lambda\in 2^{\mathbb Z}$, $\mathtt q\in\mathcal Q_\mu$, $d\in 2^{\mathbb Z}$:
$$
P_\lambda = \rho\left(\frac{|-i\nabla|}{\lambda}\right),\quad P_q = \rho_q(-i\nabla),\quad C^{\theta}_d = \rho\left(\frac{|-i\partial_t+\theta|\nabla||}{d}\right).
$$
We also define $C^\theta_{\le d}=\sum_{\delta\le d}C^\theta_\delta$ and $C^\theta_{\ge d}$ is defined in the similar way. 
Given $0<\alpha\lesssim1$, we define $\mathcal C_\alpha$ to be a collection of finitely overlapping caps of radius $\alpha$ on the sphere $\mathbb S^2$. If $\kappa\in\mathcal C_\alpha$, we let $\omega_\kappa$ be the centre of the cap $\kappa$. Then we define $\{\rho_\kappa\}_{\kappa\in\mathcal C_\alpha}$ to be a smooth partition of unity subordinate to the conic sectors $\{ \xi\neq0 , \frac{\xi}{|\xi|}\in\kappa \}$ and denote the angular Fourier localisation multipliers by
$
R_\kappa = \rho_\kappa(-i\nabla).
$
\subsection{Analysis on the sphere}\label{an-sph}
We introduce some basic facts from harmonic analysis on the unit sphere. The most of ingredients can be found in \cite{candyherr,ster}. We also refer the readers to \cite{steinweiss} for more systematic introduction to the spherical harmonics. We let ${\mathfrak S}_{\ell}$ be the set of homogeneous harmonic polynomial of degree $\ell$. Then define $\{ Y_{\ell, m} \}_{m = -\ell}^{\ell}$ a set of orthonormal basis for ${\mathfrak S}_{\ell}$, with respect to the inner product:
\begin{align}
\langle Y_{\ell, m}, Y_{\ell',m'}\rangle_{L^2_\omega(\mathbb S^2)} = \int_{\mathbb S^2}{Y_{\ell, m}(\omega)} \overline{Y_{\ell',m'}(\omega)}\,d\omega.
\end{align}
Given $f\in L^2_x(\mathbb R^3)$, we have the orthogonal decomposition as follow:
\begin{align}
f(x) = \sum_{\ell}\sum_{m=-\ell}^{\ell}\langle f(|x|\omega), Y_{\ell, m}(\omega)\rangle_{L^2_\omega(\mathbb S^2)} Y_{\ell, m}\big(\frac{x}{|x|}\big).
\end{align}
For a dyadic number $N>1$, we define the spherical dyadic decompositions by
\begin{align}
H_N(f)(x) & = 	\sum_{\ell}\sum_{m = -\ell}^{\ell}\chi_{[N, 2N)}(\ell)\langle f(|x|\omega), Y_{\ell, m}(\omega)\rangle_{L^2_\omega(\mathbb S^2)}Y_{\ell ,m}\big(\frac{x}{|x|}\big), \\
H_1(f)(x) & = \sum_{\ell = 0, 1}\sum_{m=-\ell}^{\ell}\langle f(|x|\omega), Y_{\ell, m}(\omega)\rangle_{L^2_\omega(\mathbb S^2)}Y_{\ell, m}\big(\frac{x}{|x|}\big),
\end{align}
where $\chi_I$ is the characteristic function supported on the interval $I$, i.e., $\chi(s)=1$ for $s\in I$ and $\chi(s)=0$ when $s\notin I$.
Since $-\Delta_{\mathbb S^2}Y_{\ell, m} = \ell(\ell+1)Y_{\ell, m}$, by orthogonality one can readily get
$$\|\langle\Omega\rangle^\sigma f\|_{L^2_\omega({\mathbb S^2})} \approx \left\|\sum_{N\in2^{\mathbb N \cup\{0\}}}N^\sigma H_Nf\right\|_{L^2_\omega({\mathbb S^2})}.$$
\begin{prop}[Theorem 3.10 of \cite{steinweiss}]\label{sph-fourier}
	Suppose that $f\in L^2_x(\mathbb R^3)$ has the form $f(x) = f_0(|x|)Y_{\ell,m}(\frac{x}{|x|})$, where $Y_{\ell,m}\in\mathfrak S_\ell$. Then the Fourier transform $\widehat f$ of $f$ has the form $\widehat f(\xi)=F_0(|\xi|)Y_{\ell,m}(\frac{\xi}{|\xi|})$, where 
	$$
	F_0(r) = 2\pi i^{-\ell}r^{-\frac{2\ell+1}{2}}\int_0^\infty f_0(s) J_{\frac{2\ell+1}{2}}(2\pi rs) s^{\frac{2\ell+1}{2}}\,ds.
	$$
	Here $J_m(r)$ is the Bessel functions, whose asymptotic behaviour satisfies
	\begin{align*}
	J_m(r) &\approx r^m, \quad r\lesssim1, \\
	J_m(r) &= O(r^{-\frac12}), \quad r\gg1.	
	\end{align*}
\end{prop}
In this paper we do not use the explicit formular of the Fourier transform of the radial part of given functions. The key point is that the set $\mathfrak S_\ell$ is closed under the Fourier transforms. Now we introduce the Sobolev embedding on the unit sphere $\mathbb S^2\subset\mathbb R^3$.
\begin{prop}\label{bernstein-sph}
Let $n\ge2$. Let $f$ be a test function on the unit sphere $\mathbb S^{n-1}\subset\mathbb R^n$. For $2\le q\le p<\infty$, we have
$$
\|H_Nf\|_{L^p_\omega(\mathbb S^{n-1})} \lesssim N^{(n-1)(\frac1q-\frac1p)}\|H_N f\|_{L^q_\omega(\mathbb S^{n-1})}.
$$
\end{prop}


\begin{lem}[Lemma 7.1. of \cite{candyherr}]\label{sph-ortho}
Let $N\ge1$. Then $H_N$ is uniformly bounded on $L^p(\mathbb R^3)$ in $N$, and $H_N$ commutes with all radial Fourier multipliers. Moreover, if $N'\ge1$, then either $N\approx N'$ or
$$
H_N\Pi_\theta H_{N'}=0.
$$	
\end{lem}
By Lemma \ref{sph-ortho}, we see that $H_N$ commutes with the $P_\lambda$ and $C^\theta_d$ multipliers since we can write $C_d^\theta=e^{-\theta it|\nabla|}\rho(-\frac{i\partial_t}{d})e^{\theta it|\nabla|}$. On the other hand, we note that $H_N$ does not commute with the cube and cap localisation operators $R_\kappa$ and $P_q$, which are non-radial. Obviously, $H_N$ does not commute with $\Pi_\pm$. However we can still enjoy the orthogonality of the projections $\Pi_\theta$ in the presence of $H_N$ by using the followng.
\begin{lem}\label{comm-HN}
Let $\theta\in\{+,-\}$. If $\Pi_{\theta}\phi = 0$ in $L_x^2$, then $\sum_{N \ge 1}\Pi_{\theta}H_N\phi = 0$.
\end{lem}
\begin{proof}
We first observe that for any compactly supported smooth radial $a_{\ell,m}$
$$
\Pi_\theta(a_{\ell, m}Y_{\ell, m}) = \frac12(a_{\ell, m}Y_{\ell ,m}+\theta \gamma^0\gamma^j \partial_j( b_{\ell, m}Y_{\ell, m}),
$$
where $b_{\ell ,m} = |\nabla|^{-1}a_{\ell, m}$.
Using the identity $\partial_j = \frac{x_j}{r} \partial_r + \sum_{k = 1}^3x_k\Omega_{jk}$, one has
$$
\frac{x_j}{r} \partial_r(b_{\ell, m}Y_{\ell, m}) = \partial_rb_{\ell, m}\frac{x_j}{r}Y_{\ell, m}.
$$
In view of the recursion formula of associated Legendre polynomials, $\frac{x_j}{r}Y_{\ell, m}$ is the linear combination of spherical harmonic functions of degree $\ell \pm 1$.
Meanwhile, $\sum_{k = 1}^3x_k\Omega_{jk}Y_{\ell, m}(\frac{x}{r}) = r^{-\ell}\sum_{k = 1}^3x_k\Omega_{jk}(Y_{\ell,m}(x))$ and
$\sum_{k = 1}^3x_k\Omega_{jk}(Y_{\ell,m}(x)) + 3x_jY_{\ell, m}(x)$ is a harmonic polynomial of degree $\ell+1$. Hence we deduce that
$\sum_{k = 1}^3x_k\Omega_{jk}Y_{\ell ,m}$ is a linear combination of spherical harmonic functions of degree $\ell \pm 1$.

Now in general, if $\Pi_{\theta}\phi = 0$ in $L_x^2$, we have
$\langle \Pi_{\theta}\phi, Y_{\ell, m}\rangle_{L^2_\omega(\mathbb S^2)} = 0$ for all $\ell \ge 0, -\ell \le m \le \ell$.
Using the decomposition $\phi = \sum_\ell \phi_\ell$, where $ \phi_\ell = \sum_{m  = -\ell}^\ell a_{\ell, m}Y_{\ell, m}$, from the above argument it follows that for all $\ell \ge 1$
$$
\Pi_\theta(\phi_\ell) = \phi^{\theta, 1}_{\ell-1} + \phi_\ell^{\theta, 2} + \phi_{\ell+1}^{\theta, 3},
$$
where $\phi_j^{\theta, k}$ is the linear combination of functions $c_{j,m'}^{\theta, k}(r)Y_{j,m'}\;\;(-j \le m' \le j)$. Since $\phi_\ell \to 0$ in $L_x^2$ as $\ell \to \infty$ and $\Pi_\theta$ is the bounded operator in $L_x^2$, by orthogonality we see that
$$
\lim_{\ell \to \infty} \phi^{\theta, k}_{\ell} = 0 \;\;\mbox{in}\;\;L_x^2.
$$
Since $\langle \Pi_{\theta}\phi, Y_{\ell, m}\rangle_{L^2_\omega(\mathbb S^2)} = 0$,
$$
\phi^{\pm, 3}_{\ell} + \phi_\ell^{\pm, 2} + \phi_{\ell}^{\pm, 1} = 0
$$
for all $\ell \ge 1$. This implies that
$$
\Pi_{\theta}H_N\phi = \phi_{N-1}^{\theta, 1} + \phi_N^{\theta,2} + \phi_{N+1}^{\theta, 1} + \phi_{2N-2}^{\theta, 3} + \phi_{2N-1}^{\theta, 2} + \phi_{2N}^{\theta, 3}
$$
and hence
$$
\sum_{N\ge 1 }\Pi_\theta H_N\phi = \lim_{N \to \infty}(\phi_{2N-2}^{\theta, 3} + \phi_{2N-1}^{\theta, 2} + \phi_{2N}^{\theta, 3}) = 0\;\;\mbox{in}\;\;L_x^2.
$$	
\end{proof}

\subsection{Adapted function spaces}\label{ftn-sp}
We discuss the basic properties of function spaces of $U^p$ and $V^p$ type. We refer the readers to \cite{haheko,kochtavi} for more details. Let $\mathcal I$ be the set of finite partitions $-\infty=t_0<t_1<\cdots<t_K=\infty$ and let $1\le p<\infty$.
\begin{defn}
A function $a:\mathbb R\rightarrow L^2_x$ is called a $U^p$-atom if there exists a decomposition 
$$
a=\sum_{j=1}^K\chi_{[t_k-1,t_k)}(t)f_{j-1}
$$
 with
$$
\{f_j\}_{j=0}^{K-1}\subset L^2_x,\ \sum_{j=0}^{K-1}\|f_j\|_{L^2_x}^p=1,\ f_0=0.
$$
Furthermore, we define the atomic Banach space 
$$
U^p := \left\{ u=\sum_{j=1}^\infty \lambda_ja_j : a_j \, U^p\textrm{-atom},\ \lambda_j\in\mathbb C \textrm{ such that } \sum_{j=1}^\infty|\lambda_j|<\infty  \right\}
$$
with the induced norm
$$
\|u\|_{U^p} := \inf\left\{ \sum_{j=1}^\infty|\lambda_j| : u=\sum_{j=1}^\infty \lambda_ja_j,\,\lambda_j\in\mathbb C,\, a_j \, U^p\textrm{-atom}   \right\}.
$$
\end{defn}
We list some basic properties of $U^p$ spaces. 
\begin{prop}[Proposition 2.2 of \cite{haheko}]
Let $1\le p<q<\infty$.
\begin{enumerate}
\item $U^p$ is a Banach space.
\item The embeddings $U^p\subset U^q\subset L^\infty(\mathbb R;L^2_x)$ are continuous.
\item For $u\in U^p$, $u$ is right-continuous.
\end{enumerate}
\end{prop}
\noindent We also define the space $U^p_\theta$ to be the set of all $u\in\mathbb R\rightarrow L^2_x$ such that $e^{\theta it|\nabla|}u\in U^p$ with the obvious norm
$
\|u\|_{U^p_\theta} := \|e^{\theta it|\nabla|}u\|_{U^p}.
$
We define the $2$-variation of $v$ to be
$$
|v|_{V^2} = \sup_{ \{t_k\}_{k=0}^K\in\mathcal I } \left( \sum_{k=0}^K\|v(t_k)-v(t_{k-1})\|_{L^2_x}^2 \right)^\frac12
$$
Then the Banach space $V^2$ can be defined to be all right continuous functions $v:\mathbb R\rightarrow L^2_x$ such that the quantity
$$
\|v\|_{V^2} = \|v\|_{L^\infty_tL^2_x} + |v|_{V^2}
$$
is finite. Set $\|u\|_{V^2_\theta}=\|e^{	\theta it|\nabla|}u\|_{V^2}$. We recall basic properties of $V^2_\theta$ space from \cite{candyherr, candyherr1, haheko}.
\begin{lem}[Lemma 2.3 of \cite{cyang}]
Let $2\le p < q <\infty$. The embedding $V^p\subset U^q$ is continuous. In particular, we have $\|u\|_{U^q_\theta}\lesssim \|u\|_{V^p_\theta}$.	
\end{lem}
We shall use the following lemma to prove the scattering result.
\begin{lem}[Lemma 7.4 of \cite{candyherr}]\label{v-scatter}
	Let $u\in V^2_\theta$. Then there exists $f\in L^2_x$ such that $\|u(t)-e^{-\theta it|\nabla|}f\|_{L^2_x}\rightarrow0$ as $t\rightarrow\pm\infty$.
\end{lem}
Recall the modulation-localisation $C^\theta_d$. The following lemma is on a simple bound in the high-modulation region.
\begin{lem}[Corollary 2.18 of \cite{haheko}]
Let $2\le q\le\infty$. For $d\in2^{\mathbb Z}$ and $\theta \in \{+, -\}$, we have
\begin{align}\label{bdd-high-mod}
\begin{aligned}
\|C^{\theta}_d u\|_{L^q_tL^2_x} \lesssim d^{-\frac1q}\|u\|_{V^2_\theta},
\end{aligned}
\end{align}
\end{lem}
\begin{lem}[Lemma 7.3. of \cite{candyherr}]\label{energy-ineq}
Let $F\in L^\infty_tL^2_x$, and suppose that
$$
\sup_{\|P_\lambda H_Nv\|_{V^2_\theta}\lesssim1}\left|\int_{\mathbb R} \langle P_\lambda H_Nv(t),F(t)\rangle_{L^2_x}\,dt \right| <\infty.
$$	
If $u\in C(\mathbb R,L^2_x)$ satisfies $(-i\partial_t+\theta|\nabla|) u=F$, then $P_\lambda H_Nu\in V^2_\theta$ and we have the bound
\begin{align}\label{v-dual}
\|P_\lambda H_Nu\|_{V^2_\theta} \lesssim \|P_\lambda H_Nu(0)\|_{L^2_x} + \sup_{\|P_\lambda H_Nv\|_{V^2_\theta}\lesssim1}\left|\int_{\mathbb R} \langle P_\lambda H_Nv(t),F(t)\rangle_{L^2_x}\,dt \right|.
\end{align}
\end{lem}


We define the Banach space associated with the homogeneous Sobolev space to be the set
$$
 F^{s,\sigma}_\theta = \left\{ u\in C(\mathbb R;\langle\Omega\rangle^{-\sigma}\dot H^s): \|u\|_{ F^{s,\sigma}_\theta}<\infty \right\},
$$
where the norm is defined by
$$
\|u\|_{ F^{s,\sigma}_\theta} = \bigg( \sum_{\lambda\in2^{\mathbb Z}}\sum_{N\ge1}\lambda^{2s}N^{2\sigma}\|P_\lambda H_Nu\|_{V^2_{\theta}}^2 \bigg)^\frac12.
$$

\subsection{Auxiliary estimates}
To reveal null form in the nonlinearity of the system \eqref{dirac-decom}, we write
\begin{align}\label{dec-di}
\begin{aligned}
(\Pi_{\theta_1}\phi)^\dagger\gamma^0\Pi_{\theta_2}\varphi = &[(\Pi_{\theta_1}-\Pi_{\theta_1}(x))\phi]^\dagger\gamma^0\Pi_{\theta_2}\varphi+(\Pi_{\theta_1}\phi)^\dagger\gamma^0(\Pi_{\theta_2}-\Pi_{\theta_2}(y))\varphi \\
&\qquad\qquad+\phi^\dagger\Pi_{\theta_1}(x)\gamma^0\Pi_{\theta_2}(y)\varphi,
\end{aligned}
\end{align}
for any $x,y\in\mathbb R^3$. Then we have the following null-form-type bound:
\begin{align}\label{dirac-null}
|\Pi_{\theta_1}(\xi)\gamma^0\Pi_{\theta_2}(\eta)| \lesssim \angle(\theta_1\xi,\theta_2\eta).
\end{align}
To exploit the null form for the first and second terms of \eqref{dec-di}, we use the following lemma:
\begin{lem}[Lemma 8.1. of \cite{candyherr}]\label{lem-null}
Let $1<r<\infty$. If $\lambda\ge1$, $\alpha\gtrsim\lambda^{-1},\ \kappa\in\mathcal C_\alpha$, then
$$
\|(\Pi_\theta-\Pi_{\theta}(\lambda\omega(\kappa)))R_\kappa P_\lambda f\|_{L^r_x} \lesssim \alpha\|R_\kappa P_\lambda f\|_{L^r_x}.
$$	
\end{lem}
Now we introduce the classical Strichartz estimates. It is well-known that the homogeneous solutions of the wave equations satisfy the space-time estimates as $\|e^{-\theta it|\nabla|}P_\lambda f\|_{L^q_tL^r_x}\lesssim \lambda^\frac2q\|P_\lambda f\|_{L^2_x}$, provided that $\frac1q+\frac1r=\frac12$ and $2<q\le\infty$. See \cite{choozxia}. A simple use of the linear estimates give enough bound for the proof of our main theorem in the Low$\times$High $\Rightarrow$ High interactions. This is obviously because the Fourier multiplier $\widehat V_b(\xi)$ plays a role as the kernel, which yields good decay. However, in the High$\times$High $\Rightarrow$ Low interactions, the above linear estimate is not enough. Even worse, the Fourier multiplier $\widehat V_b(\xi)$ becomes a serious singularity especially when $b=0$. To overcome such a problem we shall use the refinement of the classical Strichartz estimates via smaller cube localisations. The following refined estimates can be found in \cite[Lemma 3.1]{behe}. We also refer the readers to \cite{klaitataru}, which concerns the refined estimates in general dimensional setting and decay estimates.
\begin{lem}\label{stri-cube}
Let $0<\mu\ll\lambda$. Suppose that $(q,r)$ satisfies $\frac1q+\frac1r=\frac12$. Then 
\begin{align}
\|e^{\theta it|\nabla|}P_{\mathtt q}P_\lambda f\|_{L^q_tL^r_x} & \lesssim (\mu\lambda)^\frac1{ q} \|P_\mathtt qP_{\lambda} f\|_{L^2_x}	
\end{align}
\end{lem}
\begin{proof}
The proof follows from Lemma 3.1 of 	\cite{behe}. Indeed, we let $T=P_{\mathtt q}P_{\lambda}e^{\theta it|\nabla|}$. The required estimate can be obtained by the standard $TT^*$ argument. In fact, the operator $TT^*$ is a space-time convolution operator with kernel
$$
K_{\mathtt q,\lambda}(t,x) = \int_{\mathbb R^3}e^{+\theta it|\xi|+ix\cdot\xi}\rho_\lambda^2(\xi)\rho_{\mathtt q}(\xi)\,d\xi.
$$
Then it is enough to show
$$
\|TT^*\|_{L^{q'}_tL^{r'}_x\rightarrow L^q_tL^r_x} \lesssim (\mu\lambda)^\frac2q,
$$
where $\frac1q+\frac1{q'}=1$ and $\frac1r+\frac1{r'}=1$.
In view of complex interpolation and Young's inequality and Hardy-Littlewood-Sobolev inequality, the above estimate is reduced to the following kernel bound:
$$
|K_{\mathtt q,\lambda}(t,x)| \lesssim \mu^3(1+\mu^2\lambda^{-1}|t|)^{-1}.
$$
The scaling argument gives
\begin{align*}
K_{\mathtt q,\lambda}(\lambda t,\lambda x) &= \int_{\mathbb R^3}e^{+\theta it|\xi|+ix\cdot\xi} \rho_1^2(\xi)\rho^2_{\mathtt q}(\lambda^{-1}\xi)\,d\xi \\
& := \mathcal K(s,y).	
\end{align*}
Then the remaining task is to prove
$$
|\mathcal K(s,y)| \lesssim (\mu\lambda^{-1})^3\left(1+(\mu\lambda^{-1})^2|s|\right)^{-1}.
$$
For $|s|\lesssim (\mu\lambda^{-1})^{-2}$, the bound is obvious, because the volume measure of the support of the integrand is $(\mu\lambda^{-1})^3$. If $|s|\gg (\mu\lambda^{-1})^{-2}$, then we may replace the cut-off $\rho_1^2(\xi)\rho_{\mathtt q}(\lambda^{-1}\xi)$ by a smooth cut-off $\zeta$ with respect to a thickened spherical cap of size $\mu\lambda^{-1}$. Now we let $\widetilde{\mathcal K}$ denote the corresponding kernel. We further assume that $y=(0,0,|y|)$ by rotation. By the use of spherical coordinates we write
$$
\widetilde{\mathcal K}(s,y) = \int_0^\infty\int_0^{2\pi}\int_0^\pi e^{i(|y|r\cos\omega+sr)}\zeta(\omega,\varphi,r)\sin\omega r^2 \,d\omega d\varphi dr.
$$
We may choose $\zeta(\omega,\varphi,r)=\zeta_1(\omega)\zeta_2(\varphi)\zeta_3(r)$. The stationary point of the phase of the oscillatory integral occurs only if $|y|\approx|s|$ and the cap is centered near the north pole or the south pole. Thus it suffices to consider the case when the cap is localised near the north pole, since the remaining cases yield similar (when localised near the south pole,) or even better bound by oscillation. Now we assume that $|\zeta'_1|\lesssim (\mu\lambda^{-1})^{-1}$, $\zeta_1$ is supported in an interval of length $\lesssim\mu\lambda^{-1}$ in $[0,\pi)$, and $\zeta_3$ is supported in an interval of length $\lesssim \mu\lambda^{-1}$ in $(\frac12,2)$, with $|\zeta_3'|\lesssim(\mu\lambda^{-1})^{-1}$. Then the integration by parts with respect to $\omega$ yields
\begin{align*}
\widetilde{\mathcal K}(s,y) &=\frac{i\zeta_1(0)}{|y|}\int_0^\infty\int_0^{2\pi} e^{i(|y|r+sr)}\zeta_2(\varphi)\zeta_3(r)r\,d\varphi dr \\
&\qquad - \frac{i}{|y|}\int_0^\infty\int_0^{2\pi}\int_0^\pi e^{i(|y|r\cos\omega+sr)}\zeta_1'(\omega)d\omega \zeta_2(\varphi)\zeta_3(r)r\,d\varphi dr.	
\end{align*}
Then the assumptions on $\zeta_1$ and $\zeta_3$ give
$$
|\widetilde{\mathcal K}(s,y)|\lesssim (\mu\lambda^{-1})|y|^{-1},
$$
which gives the required estimates.
\end{proof}
After the use of the refined estimates, we need to deal with the square sum to recover the $V^2_\theta$-norm, which causes a certain loss in our estimates. The following lemma says that such a loss is not harmful.
\begin{lem}[Lemma 8.6. of \cite{candyherr}]\label{square-sum}
Let $\{P_j\}_{j\in\mathcal J}$ and $\{M_j\}_{j\in\mathcal J}$ be a collection of spatial Fourier multipliers. Suppose that the symbols of $P_j$ have finite overlap, and
$$
\|M_jP_jf\|_{L^2_x} \lesssim \delta \|P_jf\|_{L^2_x}
$$	
for some $\delta>0$. Let $q>2,\ r\ge2$. Suppose that there exists $A>0$ such that for every $j$ we have the bound
$$
\|e^{-\theta it|\nabla|}P_jf\|_{L^q_tL^r_x} \le A\|P_jf\|_{L^2_x}.
$$
Then for every $\epsilon>0$, we have
$$
\left(\sum_{j\in\mathcal J}\|M_jP_jv\|_{L^q_tL^r_x}^2\right)^\frac12 \lesssim \delta|\mathcal J|^\epsilon A\|v\|_{V^2_\theta}.
$$
Here $|\mathcal J|$ is the cardinal number of the set $\mathcal J$.
\end{lem}
As a direct application of Lemma \ref{square-sum}, we shall often use the following: for $\mu\le\lambda$ and $\alpha\gtrsim\lambda^{-1}$,
$$
\bigg( \sum_{\mathtt q\in\mathcal Q_\mu}\sum_{\kappa\in\mathcal C_\alpha}\|P_\mathtt qR_\kappa P_\lambda H_N u\|_{L^4_tL^4_x}^2 \bigg)^\frac12 \lesssim (\mu\lambda)^\frac14 \left(\frac\mu\lambda\right)^{-\epsilon}\alpha^{-\epsilon}\|P_\lambda H_N u\|_{V^2_\theta}.
$$
As we are concerned with the equations \eqref{main-hartree} with $\Gamma=\gamma^0$, we cannot exploit the null structure anymore. However, in the sprit of \cite{ster}, the use of angular momentum operator substitutes for the role of null structures. Indeed, by spending an additional angular regularity one can enjoy improved space-time estimates as follows. We also refer to \cite{choslee}, which concerns improved space-time estimates for several differential operators.
\begin{prop}[Proposition 3.4 of \cite{ster}]\label{stri-ang}
For $\frac{1}{10}\ge\eta>0$, let $q_\eta=\frac{2}{1-3\eta}$. We have the improved Strichartz estimates by imposing angular regularity as follow:
\begin{align}
\|e^{\theta it|\nabla|}P_\lambda H_N f\|_{L^{q_\eta}_tL^{4}_x} \lesssim \lambda^{\frac34-\frac1{q_\eta}}N^{\frac12+\eta}\|	P_\lambda H_Nf\|_{L^2_x}.
\end{align}
\end{prop}
However, it is easily seen that the angular regularity $\langle\Omega\rangle^{\frac12+\epsilon}$ is not enough to relax the specific singularity $|\xi|^{-2}$ in the proof of Theorem \ref{gwp-hartree2}. We need to seek another way to exploit an additional angular regularity. We introduce one approach given by \cite{sterbenz2} so called an \textit{angular concentration phenomena}, which does not use the dispersion of solutions.  
\begin{lem}[Lemma 5.2 of \cite{sterbenz2}]\label{ang-con}
Let $2\le p<\infty$, and $0\le \sigma<\frac{n-1}p$. If $\lambda\in2^\mathbb Z$, $N\ge1$, $0<{\alpha}\lesssim1$, and $\kappa\in\mathcal C_{\alpha}$, then we have
\begin{align}\label{ang-con1}
\|R_\kappa P_\lambda H_N f\|_{L^p_x(\mathbb R^n)} \lesssim ({\alpha} N)^\sigma \|P_\lambda H_N f\|_{L^p_x(\mathbb R^n)}.
\end{align}
\end{lem}
\begin{proof}
By orthogonality of spherical harmonics, it is no harm to assume that $f(x) = f_0(|x|)Y_{\ell}(\frac{x}{|x|})$, where $f_0$ is a radial function whose Fourier transform is localised in an annular domain of size $\lambda$ and $Y_\ell$ is a spherical harmonic polynomial of degree $\ell$ and $N\le \ell< 2N$. It suffices to show that for $s<\frac{n-1}{2}$,
$$
\|R_\kappa  f\|_{L^2_x} \lesssim (\alpha N)^{s} \| f\|_{L^2_x}.
$$	
Then the interpolation with the trivial bound $\|R_\kappa f\|_{L^\infty_x}\lesssim \|f\|_{L^\infty_x} $ gives \eqref{ang-con1}. We let $s<\frac{n-1}{2}$. As the proof of \eqref{ang-con1}, we only use the H\"older inequality and the angular Sobolev embedding Proposition \ref{bernstein-sph}
\begin{align*}
\|R_\kappa f	\|_{L^2_x}^2 = \| \rho_\kappa \widehat f\|_{L^2_\xi}^2 &= \int_0^\infty \| \rho_\kappa \widehat f\|_{L^2_\omega(\mathbb S^{n-1})}^2 r^{n-1}\,dr \\
& \lesssim  \int_0^\infty  \|\rho_\kappa\|_{L^{\frac{n-1}{s}}_\omega}^2\|\widehat f\|_{L^{\frac{2(n-1)}{n-1-2s}}_\omega}^2 r^{n-1}\,dr \\
& \lesssim \int_0^\infty  \|\rho_\kappa\|_{L^{\frac{n-1}{s}}_\omega}^2\|\widehat f\|_{L^{2}_\omega}^2 r^{n-1}\,dr \\
& \lesssim (\alpha N)^{2s} \|f\|_{L^2_x}^2.
\end{align*}	
\end{proof}

\section{Bilinear estimates: proof of Theorem}
This section is devoted to the proof of Theorem \ref{gwp-hartree1} and Theorem \ref{gwp-hartree2}. The proof of Theorem \ref{gwp-hartree4} follows by an identical manner as the proof of Theorem \ref{gwp-hartree2}. We first define the Duhamel integral 
$$
\mathfrak I^\theta[F] = \int_0^t e^{-\theta i(t-t')|\nabla|}F(t')\,dt'.
$$
Then the integral $\mathfrak I^\theta[F]$ solves the half-wave equation
$$
(-i\partial_t+\theta|\nabla|)\mathfrak I^\theta[F] = F,
$$
with vanishing data at $t=0$. For the proof of Theorem \ref{gwp-hartree1} and Theorem \ref{gwp-hartree2} it suffices to show the following trilinear estimates: for $\sigma>0$,
\begin{align}
\|\mathfrak I^\theta[\Pi_\theta (V_b*(\psi_1^\dagger\gamma^0\psi_2)\gamma^0\psi_3)]\|_{F^{0,\sigma}_\theta} & \lesssim \|\psi_1\|_{F^{0,\sigma}_{\theta_1}}\|\psi_2\|_{F^{0,\sigma}_{\theta_2}}\|\psi_3\|_{F^{0,\sigma}_{\theta_3}}, \label{main-tri1} \\
\|\mathfrak I^\theta[\Pi_\theta (V_b*(\psi_1^\dagger\psi_2)\psi_3)]\|_{F^{0,1}_\theta} & \lesssim \|\psi_1\|_{F^{0,1}_{\theta_1}}\|\psi_2\|_{F^{0,1}_{\theta_2}}\|\psi_3\|_{F^{0,1}_{\theta_3}}. \label{main-tri2}
\end{align}
Indeed, multilinear estimates \eqref{main-tri1} and \eqref{main-tri2} together with the standard contraction argument give the global solutions to the equations \eqref{main-hartree} for $\Gamma=I_{4\times4}$ and $\Gamma=\gamma^0$, respectively, when we have the appropriate smallness condition for the initial data $\psi_0$. Moreover, the finiteness of the $V^2_\theta$-norm of the solutions implies the scattering property by an application of Lemma \ref{v-scatter}. In view of the energy inequality Lemma \ref{energy-ineq} the proof of the trilinear estimates \eqref{main-tri1} and \eqref{main-tri2} is reduced to the estimates of the following quad-linear expression:
\begin{align}\label{quart-int1}
\int_{\mathbb R^{1+3}} V_b*(\psi_1^\dagger\gamma^0\psi_2)(\psi_4^\dagger\gamma^0\psi_3)\,dxdt,
\end{align}
and
\begin{align}\label{quart-int2}
\int_{\mathbb R^{1+3}} V_b*(\psi_1^\dagger\psi_2)(\psi_4^\dagger\psi_3)\,dxdt.	
\end{align}
We let $\xi_j$ be the spatial Fourier variables of the spinor field $\psi_j$, $j=1,2,3,4$. In view of the Plancherel's theorem, we have the frequency-relations
$$
-\xi_1+\xi_2-\xi_4+\xi_3=0,
$$ 
or 
$$
\xi_0 = -\xi_1+\xi_2 = -\xi_4+\xi_3.
$$
For a moment we assume that the Fourier transforms of the spinor fields $\psi_j$ are localised in annuli of size $\lambda_j$, respectively. Then the quad-linear expression vanishes unless the following frequency-relations hold
\begin{align*}
\min\{\lambda_0,\lambda_j,\lambda_k\} & \lesssim \textrm{med}\{\lambda_0,\lambda_j,\lambda_k\} \approx \max\{\lambda_0,\lambda_j,\lambda_k\},
\end{align*}
where $\{j,k\}=\{1,2\}$, or $\{j,k\}=\{3,4\}$. We have the similar relations for the angular frequencies after the use of the spherical Littlewood-Paley projections $H_{N_j}$
\begin{align*}
\min\{N_0,N_j,N_k\}	&\lesssim \textrm{med}\{N_0,N_j,N_k\}	 \approx \max\{N_0,N_j,N_k\}	. 
\end{align*}
We further decompose the integrand via the modulation $d$, the distance to the characteristic hypersurface (or the light cone). To do this we introduce the temporal Fourier variables $\tau_j$ of the spinors $\psi_j$. The modulation of the spinor $\psi_j$ is given by $|\tau_j+\theta_j|\xi_j||$. We recall the modulation functions 
$$
\mathcal M_{\theta_{1234}}(\xi_1,\xi_2,\xi_3,\xi_4)=| -\theta_1|\xi_1|+\theta_2|\xi_2|+\theta_3|\xi_3|-\theta_4|\xi_4| |	.
$$
We pay special attention to the resonant interactions, which means that $\mathcal M_{\theta_{1234}}$ is relatively small. The modulation function is small only when\footnote{Here we ignore the high-output cases such as $\lambda_1\ll\lambda_2$ or $\lambda_3\ll\lambda_4$. As the readers will see below Lemma \ref{main-bi}, the high-output cases can be easily treated.}
\begin{enumerate}
	\item $\theta_1=\theta_2$ and $\theta_3=\theta_4$ and $\lambda_1\approx\lambda_2$ and $\lambda_3\approx\lambda_4$, \label{res-int1}
	\item  $\lambda_1\approx\lambda_2\approx\lambda_3\approx\lambda_4$ with $(\theta_1,\theta_2,\theta_3,\theta_4)=(+,-,-,+)$ or $(-,+,+,-)$. \label{res-int2}
\end{enumerate}
In the case \eqref{res-int2}, we cannot use the null structure and we get the bound as (see also Proposition 3.7 of \cite{cyang}) $$\|P_\mu(\varphi_{\lambda_1}^\dagger\gamma^0\psi_{\lambda_2})\|_{L^2_{t,x}}\lesssim\mu\|\varphi_{\lambda_1}\|_{V^2_{\theta_1}}\|\psi_{\lambda_2}\|_{V^2_{\theta_2}}.$$
The resonant interaction \eqref{res-int2} is the main drawback, which hinders one from obtaining the global solutions for the $L^2_x$-data. We overcome such an obstruction by applying extra weighted regularity in the angular variables. In what follows, we will prove for some $\mathfrak d>0$,
\begin{align*}
\|P_\mu H_N(\varphi_{\lambda_1,N_1}^\dagger\gamma^0\psi_{\lambda_2,N_2})\|_{L^2_tL^2_x} & \lesssim \mu\left(\frac\mu{\min\{\lambda_1,\lambda_2\}}\right)^{\mathfrak d}	\min\{N_1,N_2\}\|\varphi_{\lambda_1,N_1}\|_{V^2_{\theta_1}}\|\psi_{\lambda_2,N_2}\|_{V^2_{\theta_2}},	
\end{align*}
   and hence we simply combine two bound to get
   \begin{align*}
\|P_\mu H_N(\varphi_{\lambda_1,N_1}^\dagger\gamma^0\psi_{\lambda_2,N_2})\|_{L^2_tL^2_x} & \lesssim \mu\left(\frac\mu{\min\{\lambda_1,\lambda_2\}}\right)^{\frac\delta{8}}	(\min\{N_1,N_2\})^\delta\|\varphi_{\lambda_1,N_1}\|_{V^2_{\theta_1}}\|\psi_{\lambda_2,N_2}\|_{V^2_{\theta_2}},	
\end{align*}
for an arbitrarily small $\delta\ll1$. Since we have $\lambda_1\approx\lambda_2\approx\lambda_3\approx\lambda_4$ it is easy to see that the above bound is enough to prove Theorem \ref{gwp-hartree1} in the case \eqref{res-int2}. 

From now on we deal with the integrals \eqref{quart-int1} and \eqref{quart-int2} in non-resonant interactions and resonant interactions other than \eqref{res-int2}. We need to consider all possible cases depending on the relative sizes of the frequency and the modulation: $d\ll \max\{\lambda_0,\lambda_j,\lambda_k\}$ and $d\gtrsim \max\{\lambda_0,\lambda_j,\lambda_k\}$. For the latter case, which is relatively high-modulation-regime, the task is rather easy. Indeed, the use of the H\"older inequality and the bound for a high-modulation-regime \eqref{bdd-high-mod} gives the required bound to prove Theorem \ref{gwp-hartree1} in the high-modulation cases. We refer to \cite{chohlee} and omit the details. On the other hand, for the relatively low-modulation-regime, i.e., $d\ll \max\{\lambda_0,\lambda_j,\lambda_k\}$, we prove the following frequency-localised $L^2$-bilinear estimates.
\begin{lem}\label{main-bi}
Let $\epsilon>0$ be arbitrarily small number. There exists $\mathfrak d>0$ such that
\begin{align}\label{main-est1}
\begin{aligned}
	& \|P_{\lambda_0}H_{N_0}(\varphi_{\lambda_1,N_1}^\dagger	\gamma^0\psi_{\lambda_2,N_2})\|_{L^2_tL^2_x} \\
	& \qquad \lesssim \lambda_0 \left(\frac{\min\{\lambda_0,\lambda_1,\lambda_2\}}{\max\{\lambda_0,\lambda_1,\lambda_2\}}\right)^\mathfrak d (\min\{N_1,N_2\})^\epsilon \|\varphi_{\lambda_1,N_1}\|_{V^2_{\theta_1}}\|\psi_{\lambda_2,N_2}\|_{V^2_{\theta_2}},
\end{aligned}	
\end{align}
and
\begin{align}\label{main-est2}
\begin{aligned}
	& \|P_{\lambda_0}H_{N_0}(\varphi_{\lambda_1,N_1}^\dagger	\psi_{\lambda_2,N_2})\|_{L^2_tL^2_x} \\
	& \qquad \lesssim \lambda_0 \left(\frac{\min\{\lambda_0,\lambda_1,\lambda_2\}}{\max\{\lambda_0,\lambda_1,\lambda_2\}}\right)^\mathfrak d (\min\{N_1,N_2\})^{1-\epsilon} \|\varphi_{\lambda_1,N_1}\|_{V^2_{\theta_1}}\|\psi_{\lambda_2,N_2}\|_{V^2_{\theta_2}}.
\end{aligned}	
\end{align}
\end{lem}
In what follows, we only consider the High$\times$High $\Rightarrow$ Low interactions, i.e., $\lambda_0\ll \lambda_1\approx\lambda_2$. Indeed, for the case $\lambda_1\lesssim \lambda_0\approx\lambda_2$ and $\lambda_2\lesssim\lambda_0\approx\lambda_1$, a simple use of the H\"older inequality and the $L^4_{t,x}$-Strichartz estimates gives the desired bound \eqref{main-est1} and \eqref{main-est2} without exploiting an additional angular regularity. In consequence the main concern is the case $\lambda_0\ll \lambda_1\approx\lambda_2$. We first decompose the modulation as follows: 
\begin{align*}
P_{\lambda_0}H_{N_0}(\varphi_{\lambda_1,N_1}^\dagger	\gamma^0\psi_{\lambda_2,N_2}) & = \sum_{d\in 2^{\mathbb Z}} C^\theta_{d}P_{\lambda_0}H_{N_0}[(C^{\theta_1}_{\le d}\varphi_{\lambda_1,N_1})^\dagger	\gamma^0(C^{\theta_2}_{\le d}\psi_{\lambda_2,N_2})] \\
&\qquad +C^\theta_{\le d}P_{\lambda_0}H_{N_0}[(C^{\theta_1}_{ d}\varphi_{\lambda_1,N_1})^\dagger\gamma^0	(C^{\theta_2}_{\le d}\psi_{\lambda_2,N_2})]  \\
&\qquad + C^\theta_{\le d}P_{\lambda_0}H_{N_0}[(C^{\theta_1}_{\le d}\varphi_{\lambda_1,N_1})^\dagger	\gamma^0(C^{\theta_2}_{ d}\psi_{\lambda_2,N_2})] \\
& := \sum_{d\in2^{\mathbb Z}}\mathcal I_0 +\mathcal I_1+\mathcal I_2 .
\end{align*}
 The key is to exploit the null structure in the bilinear form $\varphi^\dagger\gamma^0\psi$ by the decomposition as \eqref{dec-di} and using the bound \eqref{dirac-null} and Lemma \ref{lem-null}. Then the remaining step is to apply the $L^4_{t,x}$-Strichartz estimates Lemma \ref{stri-cube}. 
\subsection{Proof of \eqref{main-est1}} We only deal with the High $\times$ High $\Rightarrow$ Low interactions, i.e., $\lambda_0\ll \lambda_1\approx\lambda_2$. From now on we put $\lambda_1=\lambda_2=\lambda$ and $\lambda_0=\mu$ and assume that $\mu\ll \lambda$. We prove the bilinear estimates \eqref{main-est1} when the modulation is relatively small, i.e., $d\ll\lambda$. Then we must have $\theta_1=\theta_2$ and hence we can exploit the null structure in the bilinear form $\varphi^\dagger\gamma^0\psi$. See also Lemma 8.7 of \cite{candyherr}. We further divide the case $d\ll\lambda$ into two subcases: $d\lesssim \mu$ and $\mu\ll d\ll \lambda$.

 We first consider $d\lesssim \mu$. We let $\alpha = (\frac{d\mu}{\lambda^2})^\frac12$. After the almost orthogonal decompositions by cubes in $\mathcal Q_\mu$ and angular sectors $\kappa\in\mathcal C_\alpha$ we exploit the null structure by using the projection operators $\Pi_\theta$. Then we use the H\"older inequality and then $L^4_{t,x}$-Strichartz estimates as follows.
\begin{align*}
\mathcal I_0 &\lesssim \|C^\theta_{d}P_{\mu}H_{N}[(C^{\theta_1}_{\le d}\varphi_{\lambda,N_1})^\dagger	\gamma^0(C^{\theta_2}_{\le d}\psi_{\lambda,N_2})]\|_{L^2_tL^2_x} \\
 & \lesssim \left(\sum_{\substack{\mathtt q_1,\mathtt q_2\in\mathcal Q_\mu \\ |\mathtt q_1-\mathtt q_2|\le2\mu}}\sum_{\substack{\kappa_1,\kappa_2\in\mathcal C_\alpha \\ |\kappa_1-\kappa_2|\le2\alpha}}\left\| C^{\theta}_d P_\mu H_N[(P_{\mathtt q_1}R_{\kappa_1}C^{\theta_1}_{\le d}\varphi_{\lambda,N_1})^\dagger\gamma^0(P_{\mathtt q_2}R_{\kappa_2}C^{\theta_2}_{\le d}\psi_{\lambda,N_2})]\right\|_{L^2_tL^2_x}^2 \right)^\frac12	 \\
 & \lesssim \alpha  \left(\sum_{\substack{\mathtt q_1,\mathtt q_2\in\mathcal Q_\mu \\ |\mathtt q_1-\mathtt q_2|\le2\mu}}\sum_{\substack{\kappa_1,\kappa_2\in\mathcal C_\alpha \\ |\kappa_1-\kappa_2|\le2\alpha}}\left\| P_{\mathtt q_1}R_{\kappa_1}C^{\theta_1}_{\le d}\varphi_{\lambda,N_1}\right\|_{L^4_tL^4_x}^2\left\|P_{\mathtt q_2}R_{\kappa_2}C^{\theta_2}_{\le d}\psi_{\lambda,N_2}\right\|_{L^4_tL^4_x}^2 \right)^\frac12	\\
 & \lesssim \alpha^{1-\epsilon} (\mu\lambda)^\frac12 \left(\frac\mu\lambda\right)^{-\epsilon} \|\varphi_{\lambda,N_1}\|_{V^2_{\theta_1}}\|\psi_{\lambda,N_2}\|_{V^2_{\theta_2}}.
\end{align*}
Then the summation $d\lesssim\mu$ gives
\begin{align*}
\sum_{d\lesssim \mu}\mathcal I_0 & \lesssim \mu \left(\frac\mu\lambda\right)^{\frac12-2\epsilon}	\|\varphi_{\lambda,N_1}\|_{V^2_{\theta_1}}\|\psi_{\lambda,N_2}\|_{V^2_{\theta_2}}.
\end{align*}
For the case $\mu\ll d\ll \lambda$, we follow the identical manner as the previous case $d\lesssim\mu$. The only difference is that we use the orthogonal decomposition by angular sectors $\kappa\in\mathcal C_{\mu\lambda^{-1}}$, since the angle between the Fourier supports of the spinors $\varphi_{\lambda,N_1}$ and $\psi_{\lambda,N_2}$ is less than $\frac\mu\lambda$. Then we see that
\begin{align*}
\mathcal I_0&\lesssim \|C^\theta_{d}P_{\mu}H_{N}[(C^{\theta_1}_{\le d}\varphi_{\lambda,N_1})^\dagger	\gamma^0(C^{\theta_2}_{\le d}\psi_{\lambda,N_2})]\|_{L^2_tL^2_x} \\
 & \lesssim \left(\sum_{\substack{\mathtt q_1,\mathtt q_2\in\mathcal Q_\mu \\ |\mathtt q_1-\mathtt q_2|\le2\mu}}\sum_{\substack{\kappa_1,\kappa_2\in\mathcal C_{\mu\lambda^{-1}} \\ |\kappa_1-\kappa_2|\le2\mu\lambda^{-1}}}\left\| C^{\theta}_d P_\mu H_N[(P_{\mathtt q_1}R_{\kappa_1}C^{\theta_1}_{\le d}\varphi_{\lambda,N_1})^\dagger\gamma^0(P_{\mathtt q_2}R_{\kappa_2}C^{\theta_2}_{\le d}\psi_{\lambda,N_2})]\right\|_{L^2_tL^2_x}^2 \right)^\frac12	 \\
 & \lesssim \frac\mu\lambda  \left(\sum_{\substack{\mathtt q_1,\mathtt q_2\in\mathcal Q_\mu \\ |\mathtt q_1-\mathtt q_2|\le2\mu}}\sum_{\substack{\kappa_1,\kappa_2\in\mathcal C_{\mu\lambda^{-1}} \\ |\kappa_1-\kappa_2|\le2\mu\lambda^{-1}}}\left\| P_{\mathtt q_1}R_{\kappa_1}C^{\theta_1}_{\le d}\varphi_{\lambda,N_1}\right\|_{L^4_tL^4_x}^2\left\|P_{\mathtt q_2}R_{\kappa_2}C^{\theta_2}_{\le d}\psi_{\lambda,N_2}\right\|_{L^4_tL^4_x}^2 \right)^\frac12	\\
 & \lesssim \left(\frac\mu\lambda\right)^{1-\epsilon} (\mu\lambda)^\frac12 \left(\frac\mu\lambda\right)^{-\epsilon} \|\varphi_{\lambda,N_1}\|_{V^2_{\theta_1}}\|\psi_{\lambda,N_2}\|_{V^2_{\theta_2}}.
\end{align*}
The loss by the summation with respect to the modulation $\mu\ll d\ll\lambda$ is only $\log(\mu/\lambda) \lesssim (\frac\mu\lambda)^\epsilon$ and hence
\begin{align*}
\sum_{\mu\ll d\ll\lambda}\mathcal I_0 & \lesssim \mu\left(\frac\mu\lambda\right)^{\frac12-3\epsilon}\|\varphi_{\lambda,N_1}\|_{V^2_{\theta_1}}\|\psi_{\lambda,N_2}\|_{V^2_{\theta_2}}. 	
\end{align*}
Note that the estimates of $\mathcal I_1$ and $\mathcal I_2$ can be obtained in the exactly same way. Hence we conclude that
\begin{align}\label{bi-est-null}
	\|P_\mu H_N(\varphi_{\lambda,N_1}^\dagger\gamma^0\psi_{\lambda,N_1})\|_{L^2_tL^2_x}  & \lesssim \mu\left(\frac\mu\lambda\right)^{\frac12-\delta}\|\varphi_{\lambda,N_1}\|_{V^2_{\theta_1}}\|\psi_{\lambda,N_2}\|_{V^2_{\theta_2}},
\end{align}
for a small $\delta\ll1$.

Now we exploit the angular regularity.  We denote the two input-frequencies by $\xi_1$ and $\xi_2$, respectively. Then the angle $\angle(\xi_1,\xi_2)$ is less than $\frac\mu\lambda$. This is our first step to exploit an additional angular regularity. We let $\alpha=\mu\lambda^{-1}$. We use the almost orthogonal decompositions by angular sectors of size $\alpha$. We let $q>2$ be slightly bigger than $2$. After an application of the Bernstein inequality and the H\"older inequality, we use the angular concentration estimates \eqref{ang-con1} and improved Strichartz estimates $L^q_tL^4_x$ for $\varphi_{\lambda,N_1}$ and the classical Strichartz estimates for $\psi_{\lambda,N_2}$ as follows 
\begin{align*}
	\|P_\mu H_N(\varphi_{\lambda,N_1}^\dagger\gamma^0\psi_{\lambda,N_1})\|_{L^2_tL^2_x} & \lesssim \left( \sum_{\substack{\kappa_1,\kappa_2\in\mathcal C_{\mu\lambda^{-1} } \\ |\kappa_1-\kappa_2|\le2\mu\lambda^{-1} }} \left\|P_\mu H_N[(R_{\kappa_1}\varphi_{\lambda,N_1})^\dagger\gamma^0(R_{\kappa_2}\psi_{\lambda,N_1})]\right\|_{L^2_tL^2_x}^2 \right)^\frac12 \\
	& \lesssim \mu^{3(\frac1q-\frac14)}\left( \sum_{\substack{\kappa_1,\kappa_2\in\mathcal C_{\mu\lambda^{-1} } \\ |\kappa_1-\kappa_2|\le2\mu\lambda^{-1} }} \left\|P_\mu H_N[(R_{\kappa_1}\varphi_{\lambda,N_1})^\dagger\gamma^0(R_{\kappa_2}\psi_{\lambda,N_1})]\right\|_{L^2_tL^{\frac{4q}{q+4}}_x}^2 \right)^\frac12 \\
	& \lesssim \mu^{3(\frac1q-\frac14)}\sup_{\kappa_1}\|R_{\kappa_1}\varphi_{\lambda,N_1}\|_{L^q_tL^4_x} \bigg( \sum_{\substack{\kappa_1,\kappa_2\in\mathcal C_{\mu\lambda^{-1} } \\ |\kappa_1-\kappa_2|\le2\mu\lambda^{-1} }} \|R_{\kappa_2}\psi_{\lambda,N_2}\|_{L^{\frac{2q}{q-2}}_tL^q_x}^2 \bigg)^\frac12 \\
	& \lesssim \mu^{3(\frac1q-\frac14)} \left( \frac\mu\lambda N_1 \right)^{\frac12-2\eta} \lambda^{\frac34-\frac1q}N_1^{\frac12+\eta}\|\varphi_{\lambda,N_1}\|_{U^q_{\theta_1}}\lambda^{1-\frac2q}\left(\frac\mu\lambda\right)^{-\epsilon}\|\psi_{\lambda,N_2}\|_{V^2_{\theta_2}} \\
	& \lesssim \mu \left(\frac\mu\lambda\right)^{\frac14-\delta}N_1 \|\varphi_{\lambda,N_1}\|_{V^2_{\theta_1}}\|\psi_{\lambda,N_2}\|_{V^2_{\theta_2}}.
\end{align*}
If $N_1\gg N_2$, we simply interchange the role of $\varphi_{\lambda,N_1}$ and $\psi_{\lambda,N_2}$ and obtain
\begin{align}\label{bi-est-ang}
\|P_\mu H_N(\varphi_{\lambda,N_1}^\dagger\gamma^0\psi_{\lambda,N_1})\|_{L^2_tL^2_x} & \lesssim 	\mu \left(\frac\mu\lambda\right)^{\frac14-\delta}\min\{N_1,N_2\} \|\varphi_{\lambda,N_1}\|_{V^2_{\theta_1}}\|\psi_{\lambda,N_2}\|_{V^2_{\theta_2}}.
\end{align}
Note that we do not need to decompose the modulation.
We combine \eqref{bi-est-null} and \eqref{bi-est-ang} to get
\begin{align}
	\|P_\mu H_N(\varphi_{\lambda,N_1}^\dagger\gamma^0\psi_{\lambda,N_1})\|_{L^2_tL^2_x} & \lesssim \mu\left(\frac\mu\lambda\right)^{\frac38}(\min\{N_1,N_2\})^\epsilon \|\varphi_{\lambda,N_1}\|_{V^2_{\theta_1}}\|\psi_{\lambda,N_2}\|_{V^2_{\theta_2}},
\end{align}
where $\epsilon>0$ is arbitrarily small number. This completes the proof of \eqref{main-est1}.
\subsection{Proof of \eqref{main-est2}}
The proof of \eqref{main-est2} follows by the identical way as the proof of \eqref{bi-est-ang}. Indeed, we have
\begin{align*}
	\|P_\mu H_N(\varphi_{\lambda,N_1}^\dagger\psi_{\lambda,N_1})\|_{L^2_tL^2_x} & \lesssim \mu \left(\frac\mu\lambda\right)^{\frac14-\delta}(\min\{N_1,N_2\})^{1-\epsilon} \|\varphi_{\lambda,N_1}\|_{V^2_{\theta_1}}\|\psi_{\lambda,N_2}\|_{V^2_{\theta_2}}.
\end{align*}
This completes the proof of Lemma \ref{main-bi}.
\section{Appendix: refined bilinear estimates}
This section is devoted to a refinement of the bilinear estimates proven in the previous section. Such a refined estimate shall be used to prove large data scattering for the equation \eqref{main-hartree} with a certain condition on the initial datum. The main purpose of the Appendix here is to prove the following:
\begin{lem}\label{main-bi-refi}
Let $\epsilon>0$ be arbitrarily small number. Let $0<\delta<1$. There exists $\mathfrak d>0$ such that
\begin{align}
\begin{aligned}
	&\|P_{\lambda_0}H_{N_0}(\varphi_{\lambda_1,N_1}^\dagger\gamma^0\psi_{\lambda_2,N_2})\|_{L^2_tL^2_x} \\
	& \lesssim \lambda_0 \left(\frac{\min\{\lambda_0,\lambda_1,\lambda_2\}}{\max\{\lambda_0,\lambda_1,\lambda_2\}}\right)^\mathfrak d (\min\{N_1,N_2\})^{\epsilon} (\|\varphi_{\lambda_1,N_1}\|_{V^2_{\theta_1}}\|\psi_{\lambda_2,N_2}\|_{V^2_{\theta_2}})^{1-\delta} \\
	& \qquad\qquad\qquad\times(\lambda_1^{-\frac12}\lambda_2^{-\frac12}\|\varphi_{\lambda_1,N_1}\|_{L^4_tL^4_x}|\psi_{\lambda_2,N_2}\|_{L^4_tL^4_x})^\delta,
\end{aligned}	
\end{align}
	and
	\begin{align}
\begin{aligned}
	&\|P_{\lambda_0}H_{N_0}(\varphi_{\lambda_1,N_1}^\dagger\psi_{\lambda_2,N_2})\|_{L^2_tL^2_x} \\
	& \lesssim \lambda_0 \left(\frac{\min\{\lambda_0,\lambda_1,\lambda_2\}}{\max\{\lambda_0,\lambda_1,\lambda_2\}}\right)^\mathfrak d (\min\{N_1,N_2\})^{\epsilon} (\|\varphi_{\lambda_1,N_1}\|_{V^2_{\theta_1}}\|\psi_{\lambda_2,N_2}\|_{V^2_{\theta_2}})^{1-\delta} \\
	& \qquad\qquad\qquad\times(\lambda_1^{-\frac12}\lambda_2^{-\frac12}\|\varphi_{\lambda_1,N_1}\|_{L^4_tL^4_x}|\psi_{\lambda_2,N_2}\|_{L^4_tL^4_x})^\delta.
\end{aligned}	
\end{align}
\end{lem}
Then Lemma \ref{main-bi-refi} implies that the global-in-time solutions to the equation \eqref{main-hartree} with a large initial data, provided that a particular dispersive norm $\|u\|_{\mathbf D^{-\frac12,\sigma}}$ of the solutions given by
$$
\|u\|_{\mathbf D^{s,\sigma}}^2 = \sum_{N\ge1}N^{2\sigma}\||\nabla|^sH_N u\|_{L^4_tL^4_x}^2
$$
remains bounded as the solutions evolve in time. To avoid the repetitive task, which is already seen in previous works, instead of presenting the explicit statement and its proof, we  refer the readers to \cite{candyherr1,chohlee1} for the proof of conditional large-data scattering. In what follows, we focus on the proof of Lemma \ref{main-bi-refi}. 
We recall the important frequency-cases which result in the resonant interactions:
\begin{enumerate}
	\item $\theta_1=\theta_2$ and $\theta_3=\theta_4$ and $\lambda_1\approx\lambda_2$ and $\lambda_3\approx\lambda_4$, 
	\item  $\lambda_1\approx\lambda_2\approx\lambda_3\approx\lambda_4$ with $(\theta_1,\theta_2,\theta_3,\theta_4)=(+,-,-,+)$ or $(-,+,+,-)$. 
\end{enumerate}
We first consider the case \eqref{res-int2}. We interpolate the bilinear estimates
$$
\|P_\mu H_N(\varphi_{\lambda,N_1}^\dagger\gamma^0\psi_{\lambda,N_2})\|_{L^2_tL^2_x} \lesssim \mu \|\varphi_{\lambda,N_1}\|_{V^2_{\theta_1}}\|\psi_{\lambda,N_2}\|_{V^2_{\theta_2}}
$$
and the trivial bound
$$
\|P_\mu H_N(\varphi_{\lambda,N_1}^\dagger\gamma^0\psi_{\lambda,N_2})\|_{L^2_tL^2_x} \lesssim \lambda (\lambda^{-1} \|\varphi_{\lambda,N_1}\|_{L^4_tL^4_x}\|\psi_{\lambda,N_2}\|_{L^4_tL^4_x} )
$$
to get
\begin{align}\label{bi-est-res2c}
\begin{aligned}
	&\|P_\mu H_N(\varphi_{\lambda,N_1}^\dagger\gamma^0\psi_{\lambda,N_2})\|_{L^2_tL^2_x} \\
	&\lesssim \mu\left(\frac\mu\lambda\right)^{-\delta} (\|\varphi_{\lambda,N_1}\|_{V^2_{\theta_1}}\|\psi_{\lambda,N_2}\|_{V^2_{\theta_2}})^{1-\delta}(\lambda^{-1} \|\varphi_{\lambda,N_1}\|_{L^4_tL^4_x}\|\psi_{\lambda,N_2}\|_{L^4_tL^4_x} )^\delta.
	\end{aligned}
\end{align}
We combine two bounds \eqref{bi-est-res2c} and \eqref{bi-ang-est-ref} as
\begin{align*}
	&\|P_\mu H_N(\varphi_{\lambda,N_1}^\dagger\gamma^0\psi_{\lambda,N_2})\|_{L^2_tL^2_x} \\
	& \lesssim |\eqref{bi-est-res2c}|^{1-8\delta}|\eqref{bi-ang-est-ref}|^{8\delta} \\
	& \lesssim \mu\left(\frac\mu\lambda\right)^\delta (\min\{N_1,N_2\})^{8\delta}(\|\varphi_{\lambda,N_1}\|_{V^2_{\theta_1}}\|\psi_{\lambda,N_2}\|_{V^2_{\theta_2}})^{1-\delta} \\
	& \qquad\qquad\times (\lambda^{-1} \|\varphi_{\lambda,N_1}\|_{L^4_tL^4_x}\|\psi_{\lambda,N_2}\|_{L^4_tL^4_x} )^\delta.
\end{align*}
Note that we can choose $\delta$ arbitrarily small. This gives the proof of Lemma \ref{main-bi-refi} in the resonant case \eqref{res-int2}. Now we exclusively consider the resonant interaction \eqref{res-int1}. The proof is very similar as the previous section. The only difference is to apply the following square-summation estimates. See also \cite{candyherr1}. 
\begin{align}
\left( \sum_{\mathtt q\in\mathcal Q_\mu}\sum_{\kappa\in\mathcal C_\alpha}\|P_{\mathtt q}R_\kappa \varphi_{\lambda,N}\|_{L^4_tL^4_x}^2 \right)^\frac12 & \lesssim \alpha^{-2\delta}\left(\frac\mu\lambda\right)^{-2\delta}(\mu\lambda)^\frac14 \|\varphi_{\lambda,N}\|_{V^2_{\theta}}^{1-\delta}(\lambda^{-\frac12}\|\varphi_{\lambda,N}\|_{L^4_tL^4_x})^\delta.	
\end{align}
As we have done in the previous section, we decompose the bilinear form into the modulation. When the size of the modulation $d$ is relatively higher than the frequency, i.e., $d\gtrsim\lambda$, we directly deal with the quad-linear expression \eqref{quart-int1}. This case is rather easier than other cases. We omit it and refer to \cite{chohlee1} for details. On the other hand, if $d\lesssim\mu$ then the use of orthogonal decompositions into angular sectors and cubes together with the null form bound yields
\begin{align*}
\mathcal I_0 
 & \lesssim \left(\sum_{\substack{\mathtt q_1,\mathtt q_2\in\mathcal Q_\mu \\ |\mathtt q_1-\mathtt q_2|\le2\mu}}\sum_{\substack{\kappa_1,\kappa_2\in\mathcal C_\alpha \\ |\kappa_1-\kappa_2|\le2\alpha}}\left\| C^{\theta}_d P_\mu H_N[(P_{\mathtt q_1}R_{\kappa_1}C^{\theta_1}_{\le d}\varphi_{\lambda,N_1})^\dagger\gamma^0(P_{\mathtt q_2}R_{\kappa_2}C^{\theta_2}_{\le d}\psi_{\lambda,N_2})]\right\|_{L^2_tL^2_x}^2 \right)^\frac12	 \\
 & \lesssim \alpha  \left(\sum_{\substack{\mathtt q_1,\mathtt q_2\in\mathcal Q_\mu \\ |\mathtt q_1-\mathtt q_2|\le2\mu}}\sum_{\substack{\kappa_1,\kappa_2\in\mathcal C_\alpha \\ |\kappa_1-\kappa_2|\le2\alpha}}\left\| P_{\mathtt q_1}R_{\kappa_1}C^{\theta_1}_{\le d}\varphi_{\lambda,N_1}\right\|_{L^4_tL^4_x}^2\left\|P_{\mathtt q_2}R_{\kappa_2}C^{\theta_2}_{\le d}\psi_{\lambda,N_2}\right\|_{L^4_tL^4_x}^2 \right)^\frac12	\\
 & \lesssim \alpha^{1-4\delta} (\mu\lambda)^\frac12 \left(\frac\mu\lambda\right)^{-4\delta} \|\varphi_{\lambda,N_1}\|_{V^2_{\theta_1}}^{1-\delta}(\lambda^{-\frac12}\|\varphi_{\lambda,N_1}\|_{L^4_tL^4_x})^\delta \|\psi_{\lambda,N_1}\|_{V^2_{\theta_1}}^{1-\delta}(\lambda^{-\frac12}\|\psi_{\lambda,N_1}\|_{L^4_tL^4_x})^\delta.
\end{align*}
Then we have
\begin{align*}
\sum_{d\lesssim\mu}\mathcal I_0 & \lesssim \mu\left(\frac\mu\lambda\right)^{\frac12-8\delta}	\|\varphi_{\lambda,N_1}\|_{V^2_{\theta_1}}^{1-\delta}(\lambda^{-\frac12}\|\varphi_{\lambda,N_1}\|_{L^4_tL^4_x})^\delta \|\psi_{\lambda,N_1}\|_{V^2_{\theta_1}}^{1-\delta}(\lambda^{-\frac12}\|\psi_{\lambda,N_1}\|_{L^4_tL^4_x})^\delta.
\end{align*}
If $\mu\ll d\ll\lambda$, we follow the similar approach as the case $d\lesssim\mu$. Indeed, we have
\begin{align*}
\mathcal I_0
 & \lesssim \frac\mu\lambda  \left(\sum_{\substack{\mathtt q_1,\mathtt q_2\in\mathcal Q_\mu \\ |\mathtt q_1-\mathtt q_2|\le2\mu}}\sum_{\substack{\kappa_1,\kappa_2\in\mathcal C_{\mu\lambda^{-1}} \\ |\kappa_1-\kappa_2|\le2\mu\lambda^{-1}}}\left\| P_{\mathtt q_1}R_{\kappa_1}C^{\theta_1}_{\le d}\varphi_{\lambda,N_1}\right\|_{L^4_tL^4_x}^2\left\|P_{\mathtt q_2}R_{\kappa_2}C^{\theta_2}_{\le d}\psi_{\lambda,N_2}\right\|_{L^4_tL^4_x}^2 \right)^\frac12	\\
 & \lesssim \left(\frac\mu\lambda\right)^{1-4\delta} (\mu\lambda)^\frac12 \left(\frac\mu\lambda\right)^{-4\delta} \|\varphi_{\lambda,N_1}\|_{V^2_{\theta_1}}^{1-\delta}(\lambda^{-\frac12}\|\varphi_{\lambda,N_1}\|_{L^4_tL^4_x})^\delta \|\psi_{\lambda,N_1}\|_{V^2_{\theta_1}}^{1-\delta}(\lambda^{-\frac12}\|\psi_{\lambda,N_1}\|_{L^4_tL^4_x})^\delta.
\end{align*}
Then
\begin{align*}
\sum_{\mu\ll d\ll\lambda}\mathcal I_0 	& \lesssim \mu\left(\frac\mu\lambda\right)^{\frac12-9\delta}	\|\varphi_{\lambda,N_1}\|_{V^2_{\theta_1}}^{1-\delta}(\lambda^{-\frac12}\|\varphi_{\lambda,N_1}\|_{L^4_tL^4_x})^\delta \|\psi_{\lambda,N_1}\|_{V^2_{\theta_1}}^{1-\delta}(\lambda^{-\frac12}\|\psi_{\lambda,N_1}\|_{L^4_tL^4_x})^\delta.
\end{align*}
Hence we conclude that
\begin{align*}
\sum_{d\ll \lambda}\mathcal I_0+\mathcal I_1+\mathcal I_2	& \lesssim \mu\left(\frac\mu\lambda\right)^{\frac12-9\delta}	\|\varphi_{\lambda,N_1}\|_{V^2_{\theta_1}}^{1-\delta}(\lambda^{-\frac12}\|\varphi_{\lambda,N_1}\|_{L^4_tL^4_x})^\delta \|\psi_{\lambda,N_1}\|_{V^2_{\theta_1}}^{1-\delta}(\lambda^{-\frac12}\|\psi_{\lambda,N_1}\|_{L^4_tL^4_x})^\delta.
\end{align*}
Now we shall exploit the angular regularity. For a small $\delta\ll1$, we let $ \frac12-\frac1{10}<\frac1{q'}<\frac1q<\frac12$ so that
$
\frac{1}{q'} = \frac{1-\delta}{q}+\frac\delta4$. After an application of orthogonal decompositions of conic sectors of size $\frac\mu\lambda$, we use in order the Bernstein inequality, H\"older inequality, the convexity of the $L^p_t$-spaces, angular conentration estimates and then the Strichartz estimates to get 
\begin{align*}
\|P_\mu H_N(\varphi_{\lambda,N_1}^\dagger\gamma^0\psi_{\lambda,N_1})\|_{L^2_tL^2_x} & \lesssim \left( \sum_{\substack{\kappa_1,\kappa_2\in\mathcal C_{\mu\lambda^{-1} } \\ |\kappa_1-\kappa_2|\le2\mu\lambda^{-1} }} \left\|P_\mu H_N[(R_{\kappa_1}\varphi_{\lambda,N_1})^\dagger\gamma^0(R_{\kappa_2}\psi_{\lambda,N_1})]\right\|_{L^2_tL^2_x}^2 \right)^\frac12 \\
	& \lesssim \mu^{3(\frac1{q'}-\frac14)}\left( \sum_{\substack{\kappa_1,\kappa_2\in\mathcal C_{\mu\lambda^{-1} } \\ |\kappa_1-\kappa_2|\le2\mu\lambda^{-1} }} \left\|P_\mu H_N[(R_{\kappa_1}\varphi_{\lambda,N_1})^\dagger\gamma^0(R_{\kappa_2}\psi_{\lambda,N_1})]\right\|_{L^2_tL^{\frac{4q'}{q'+4}}_x}^2 \right)^\frac12 \\
	& \lesssim \mu^{3(\frac1{q'}-\frac14)}\sup_{\kappa_1}\|R_{\kappa_1}\varphi_{\lambda,N_1}\|_{L^{q'}_tL^4_x} \bigg( \sum_{\substack{\kappa_1,\kappa_2\in\mathcal C_{\mu\lambda^{-1} } \\ |\kappa_1-\kappa_2|\le2\mu\lambda^{-1} }} \|R_{\kappa_2}\psi_{\lambda,N_2}\|_{L^{\frac{2q'}{q'-2}}_tL^{q'}_x}^2 \bigg)^\frac12	\\
	& \lesssim \mu^{3(\frac1{q'}-\frac14)}\left(\sup_{\kappa_1}\|R_{\kappa_1}\varphi_{\lambda,N_1}\|_{L^{q}_tL^4_x}  \right)^{1-\delta} \|\varphi_{\lambda,N_1}\|_{L^4_tL^4_x}^\delta \lambda^{1-\frac{2}{q'}}\|\psi_{\lambda,N_2}\|_{V^2_{\theta_1}} \\
	& \lesssim \mu \left(\frac\mu\lambda\right)^{\frac3{q'}-\frac54-2\delta} N_1\|\varphi_{\lambda,N_1}\|_{V^2_{\theta_1}}^{1-\delta}(\lambda^{-\frac12}\|\varphi_{\lambda,N_1}\|_{L^4_tL^4_x})^\delta \|\psi_{\lambda,N_2}\|_{V^2_{\theta_1}}.
\end{align*}
On the other hand, we also have the following trivial bound using the $L^4_{t,x}$-Strichartz estimates
\begin{align*}
	\|P_\mu H_N(\varphi_{\lambda,N_1}^\dagger\gamma^0\psi_{\lambda,N_1})\|_{L^2_tL^2_x} &  \lesssim \|\varphi_{\lambda,N_1}\|_{L^4_tL^4_x} \|\psi_{\lambda,N_1}\|_{L^4_tL^4_x} \\
	& \lesssim \lambda^\frac12\|\varphi_{\lambda,N_1}\|_{V^2_{\theta_1}}^{1-\delta}\|(\lambda^{-\frac12}\|\varphi_{\lambda,N_1}\|_{L^4_tL^4_x})^\delta \|\psi_{\lambda,N_1}\|_{L^4_tL^4_x} \\
	& =\lambda\|\varphi_{\lambda,N_1}\|_{V^2_{\theta_1}}^{1-\delta}(\lambda^{-\frac12}\|\varphi_{\lambda,N_1}\|_{L^4_tL^4_x})^\delta \lambda^{-\frac12}\|\psi_{\lambda,N_1}\|_{L^4_tL^4_x} .
\end{align*}
By an interpolation of two bounds we finally have
\begin{align*}
	\|P_\mu H_N(\varphi_{\lambda,N_1}^\dagger\gamma^0\psi_{\lambda,N_1})\|_{L^2_tL^2_x} & \lesssim  \mu^{1-\delta}\lambda^\delta \left(\frac\mu\lambda\right)^{\frac18}N_1^{1-\delta}\|\varphi_{\lambda,N_1}\|_{V^2_{\theta_1}}^{1-\delta}(\lambda^{-\frac12}\|\varphi_{\lambda,N_1}\|_{L^4_tL^4_x})^\delta \\
	&\qquad\qquad \times\|\psi_{\lambda,N_1}\|_{V^2_{\theta_1}}^{1-\delta}(\lambda^{-\frac12}\|\psi_{\lambda,N_1}\|_{L^4_tL^4_x})^\delta\\
	& = \mu \left(\frac\mu\lambda\right)^{\frac18-\delta}N_1 \|\varphi_{\lambda,N_1}\|_{V^2_{\theta_1}}^{1-\delta}(\lambda^{-\frac12}\|\varphi_{\lambda,N_1}\|_{L^4_tL^4_x})^\delta \\
	& \qquad\qquad \times\|\psi_{\lambda,N_1}\|_{V^2_{\theta_1}}^{1-\delta}(\lambda^{-\frac12}\|\psi_{\lambda,N_1}\|_{L^4_tL^4_x})^\delta.
\end{align*}
If $N_1\gg N_2$ then we interchange the role of $\varphi$ and $\psi$. Hence we conclude that for some $\mathfrak d>0$
\begin{align}\label{bi-ang-est-ref}
\begin{aligned}
	&\|P_\mu H_N(\varphi_{\lambda,N_1}^\dagger\gamma^0\psi_{\lambda,N_1})\|_{L^2_tL^2_x} \\
	 & \lesssim \mu\left(\frac\mu\lambda\right)^\mathfrak d \min\{N_1,N_2\}(\|\varphi_{\lambda,N_1}\|_{V^2_{\theta_1}}\|\psi_{\lambda,N_2}\|_{V^2_{\theta_2}})^{1-\delta} \\
	 &\qquad\qquad\times(\lambda^{-1} \|\varphi_{\lambda,N_1}\|_{L^4_tL^4_x}\|\psi_{\lambda,N_2}\|_{L^4_tL^4_x} )^\delta.
	 \end{aligned}
\end{align}
Note that we do not use the specific structure of the bilinear form $\varphi^\dagger\gamma^0\psi$ in the proof of \eqref{bi-ang-est-ref}. Hence the proof of the second estimate in Lemma \ref{main-bi-refi} follows in the identical manner. This completes the proof of Lemma \ref{main-bi-refi}.
\section*{Acknowledgements}
I would like to express my gratitude to Cho, Yonggeun, who brings these problems to my attention, and also Lee, Kiyeon for his truly helpful criticisms and discussion.
This work was supported by the National Research Foundation of Korea (NRF) grant funded by the Korea government (MSIT) (NRF-2020R1A2C4002615). 

\end{document}